\documentclass[12pt,reqno]{amsart}
\usepackage{amsaddr}
\usepackage{amssymb,latexsym,amsmath,epsfig,amsthm,mathrsfs}
\usepackage{rotating}
\usepackage{graphicx}
\usepackage{amssymb}
\usepackage{lineno}
\usepackage{enumitem}
\usepackage{cite}
\usepackage[top=3cm, bottom=3cm, left=3cm, right=3cm]{geometry}
\usepackage[usenames]{color}
\usepackage[colorlinks=true,
linkcolor=blue,
filecolor=blue,
citecolor=blue]{hyperref}

\newtheorem{theorem}{Theorem}[section]
\newtheorem{corollary}[theorem]{Corollary}
\newtheorem{lemma}[theorem]{Lemma}
\theoremstyle{definition}
\newtheorem{definition}[theorem]{Definition}
\theoremstyle{definition}
\newtheorem{remark}[theorem]{Remark}

\theoremstyle{definition}

\newtheorem{example}{Example}
\numberwithin{equation}{section}

\begin{document}
	
	\title[Structure of sets with small product sets]{Structure of sets with small product sets in torsion-free groups, cyclic groups of prime orders and abelian groups}
	
	%%%%%%    Information for first author

\author[R K Mistri]{Raj Kumar Mistri$^{*}$}
\address{\em{\small Department of Mathematics, Indian Institute of Technology Bhilai\\
Durg, Chhattisgarh, India\\
e-mail: rkmistri@iitbhilai.ac.in}}

%%%%%%   Information for second author

\author[N Prajapati]{Nitesh Prajapati$^{**}$}
\address{\em{\small Department of Mathematics, Indian Institute of Technology Bhilai\\
Durg, Chhattisgarh, India\\
email: niteshp@iitbhilai.ac.in}}

\thanks{$^{*}$Corresponding author}
\thanks{$^{**}$The research of the author is supported by the UGC Fellowship (NTA Ref. No.: 211610023414)}
	
	%    General info
	

\begin{abstract}
Let $\ell$ and $m$ be positive integers with $\ell \leq m$, and let $\mathcal{A} = (A_1, \ldots, A_m)$ be a finite sequence of finite subsets of a group $G$ (not necessarily abelian), written multiplicatively. The {\it generalized product set} $\Pi^{\ell}(\mathcal{A})$ is the set of all elements of $G$ which can be represented as a product of exactly $\ell$ elements from $\ell$ distinct sets from $\mathcal{A}$ taken in any order. DeVos, Goddyn and Mohar obtained the nontrivial lower bound for the size of this product set when $G$ is abelian. The DeVos-Goddyn-Mohar Theorem is a fundamental result in additive combinatorics which unifies various results from zero-sum combinatorics and has connections with subsequence sums and sumsets. In this paper, we obtain an optimal lower bound for the size of generalized product set ${\Pi}^{\ell}(\mathcal{A})$ in torsion-free groups (not necessarily abelian), and characterize the structure of underlying sets in the sequence $\mathcal{A} = (A_1, \ldots, A_m)$ for which ${\Pi}^{\ell}(\mathcal{A})$ achieves the optimal lower bound. By slightly modifying the arguments of the proofs in the case of torsion-free groups, we derive such inverse theorems in cyclic groups of prime orders also. Our proof of these result also yields a new proof of DeVos-Goddyn-Mohar Theorem in $\mathbb{Z}_p$. Moreover, we extend these inverse results to arbitrary abelian groups. Furthermore, as an application, we generalize a theorem for subsequence sums due to Hamidoune in torsion-free groups, and obtain several other results for subsequence sums in arbitrary groups. 
\end{abstract}

%%% ----------------------------------------------------------------------
\maketitle
%%% ----------------------------------------------------------------------

\tableofcontents

\section{Introduction}
Let $A_1, \ldots, A_{\ell}$ be nonempty subsets of a group $G$ written additively. The {\it sumset} of the sets $A_1, \ldots, A_{\ell}$ is defined as
\begin{equation}\label{sumset}
A_1 + \cdots + A_{\ell} = \{a_1 + \cdots + a_{\ell} : a_i \in A_i ~\text{for}~ i = 1, \ldots, \ell\}.
\end{equation}

Sumsets play a central role in additive combinatorics and have been widely investigated in the literature. For detailed expositions and comprehensive results in this area, the reader may consult the books by Nathanson \cite{nath}, Tao and Vu \cite{tao}, Grynkiewicz \cite{gryn2013book}, and the references given therein. Given a nonemty subset $A$ of an abelian group $G$, and an element $x \in G$, we define $x + A = \{x + a: a \in A\}$. The stabilizer of the set $A$, denoted by $\mathrm{Stab}(A)$, is defined as 
\[\mathrm{Stab}(A) = \{g \in G: g + A = A\}.\]
For integers $x$ and $y$ with $x \leq y$, we denote the set $\{n \in \mathbb{Z}: x \leq n \leq y\}$ by $[x, y]$. The cardinality of a set $A$ is denoted by $|A|$. One of the well-known results in additive combinatorics is the {\it Cauchy-Davenport Theorem}, which establishes an optimal lower bound on the size of the sumset $A_1 + \cdots + A_{\ell}$. Cauchy-Davenport Theorem was first proved by Cauchy \cite{cauchy}, and it was later rediscovered by Davenport \cite{dav1, dav2}.

\begin{theorem}[Cauchy-Davenport Theorem]\label{cauchy-davenport-thm}
	Let $\ell \geq 2$ be an integer. Let $p$ be a prime, and let $A_1, \ldots, A_{\ell}$ be nonempty subset of $\mathbb{Z}_p$. Then
	\[|A_1 + \cdots + A_{\ell}| \geq \min (p, |A_1| + \cdots + |A_{\ell}| - \ell + 1).\]
\end{theorem}

A nonempty subset $A$ of $G$ is called an {\it arithmetic progression with common difference $d$} if it can be expressed as $A = \{a + jd: j \in [0, |A| - 1\}$ for some $a, d \in G$ with $d \neq 0$. The following theorem due to Vosper \cite{vosper1956} characterizes the subsets $A$ and $B$ of $\mathbb{Z}_p$ when the size of the sumset $A + B$ is optimal.
\begin{theorem}\label{vosperthm}
	Let $A$ and $B$ be nonempty subsets of $\mathbb{Z}_p$ such that $|A| \geq 2$, $|B| \geq 2$ and $|A + B| \leq p -2$, where $p$ is a prime. Then
	\[|A + B| = |A| + |B| - 1\]
	if and only if the sets $A$ and $B$ are arithmetic progressions with the same common difference.
\end{theorem}

This theorem can be generalized for the sumset of multiple sets (see Lemma \ref{vosper-gen-inv-thm1} and Lemma \ref{vosper-gen-inv-thm2} in Section \ref{sec-prime-order-group} for the precise statements).

Let $\ell$ and $m$ be positive integers with $\ell \leq m$, and let $\mathcal{A} = (A_1, \ldots, A_m)$ be a finite sequence of finite subsets of an abelian group $G$. Then the {\it generalized sumset} $\Sigma^{\ell}(\mathcal{A})$ is defined as the set of all elements of $G$ which can be represented as a sum of exactly $\ell$ elements from $\ell$ distinct sets from $\mathcal{A}$ 
\begin{equation}\label{gen-sumset} 
\Sigma^{\ell}(\mathcal{A}) = \{a_{i_1} + \cdots + a_{i_{\ell}}: a_{i_j} \in A_{i_j}~\text{for each}~ j \in [1, \ell]~\text{and}~ 1 \leq i_1 < \cdots < i_{\ell} \leq m\}.
\end{equation}

\begin{equation*}\label{gen-sumset-1} 
{\Sigma}^{\ell}(\mathcal{A}) = \bigcup_{\substack{\Lambda \subseteq [1, m]\\ |\Lambda| = \ell}} \left(\sum_{i \in \Lambda }A_i\right).
\end{equation*}

The sumset $\Sigma^{\ell}(\mathcal{A})$, defined as in \eqref{gen-sumset}, appeared in the work of DeVos, Goddyn and Mohar \cite{devos2009} who obtained a nontrivial lower bound for the size of the sumset which generalizes the following theorem of Kneser \cite{kneser1953}:
\begin{theorem} \label{kneser-thm}
Let $A$ and $B$ be nonempty finite subsets of an abelian group $G$. Then
\[|A + B| \geq |A + H| + |B + H| - |H|,\]
where $H$ is the stabilizer of $A + B$.
\end{theorem}
To state the result of DeVos, Goddyn and Mohar, we need to fix some notations. Given subsets $A$ and $Q$ of a group $G$, we define the function
\begin{equation*}
	\chi_{A}(Q) =
	\begin{cases}
	  1, & \mbox{if } A \cap Q \neq \varnothing; \\
	  0, & \mbox{if } A \cap Q = \varnothing,
	\end{cases}
\end{equation*}
where $\varnothing$ denote the empty set. In particular, if $Q = \{a\}$, then the function $\chi_{A}(Q)$ is the characteristic function of the set $A$ which is defined as
\begin{equation*}
	\chi_{A}(a) =
	\begin{cases}
	  1, & \mbox{if } a \in A; \\
	  0, & \mbox{if } a \not\in A.
	\end{cases}
\end{equation*}
Furthermore, given a sequence $\mathcal{A} = (A_1, \ldots, A_m)$ of subsets of $G$ and a set $Q \subseteq G$, we define
\begin{equation*}
  \mu_{\mathcal{A}}(Q) = \min\left(\ell, \ \sum_{i=1}^{m}\chi_{A_i}(Q) \right).
\end{equation*}
With these notations, now we can state the theorem of DeVos, Goddyn and Mohar \cite{devos2009}.
\begin{theorem}[DeVos-Goddyn-Mohar Theorem]\label{dgm-thm}
Let $\ell$ and $m$ be positive integers such that $\ell \leq m$. Let $\mathcal{A} = (A_1, \ldots, A_m)$ be a sequence of finite subsets of an abelian group $G$. Then
\begin{equation}\label{dgm-thm-lower-bound}
  |{\Sigma}^{\ell}(\mathcal{A})| \geq |H| \left(\sum_{Q \in G/H}\mu_{\mathcal{A}}(Q) - \ell + 1 \right), 
\end{equation}
where $H$ is the stabilizer of ${\Sigma}^{\ell}(\mathcal{A})$.
\end{theorem}
The DeVos-Goddyn-Mohar Theorem is a fundamental result in additive combinatorics which unifies various results from zero-sum combinatorics and has connections with subsequence sums and sumsets (see \cite{devos2009}; see also the book by Grynkiewicz \cite{gryn2013book}).

The underlying sets $A_1, \ldots, A_m$ in the sequence $\mathcal{A}$ will be called {\it extremal sets} if the equality holds in \eqref{dgm-thm-lower-bound}. One of the important problems in additive combinatorics is characterizing the structure of the extremal sets $A_1, \ldots, A_m$ corresponding to the sumset ${\Sigma}^{\ell}(\mathcal{A})$. This kind of problems is known as {\it inverse problems}. The objective of this paper is to investigate this problem. We solve this problem in torsion-free groups (not necessarily abelian), cyclic group of prime orders and also in arbitrary abelian group $G$. 

For convenience, the group $G$ is assumed to be written multiplicatively from now onward, unless otherwise stated. The multiplicative identity of $G$ will be denoted by $1$. The {\it product set} of the subsets  $A_1, \ldots, A_{\ell}$ of $G$ is defined as
\begin{equation}\label{product-set}
A_1 \cdots A_{\ell} = \{a_1 \cdots a_{\ell} : a_i \in A_i ~\text{for}~ i = 1, \ldots, \ell\}.
\end{equation}

Let $\ell$ and $m$ be positive integers with $\ell \leq m$, and let $\mathcal{A} = (A_1, \ldots, A_m)$ be a finite sequence of finite subsets of a group $G$ (not necessarily abelian). Then we define the {\it generalized product set} $\Pi^{\ell}(\mathcal{A})$ as the set of all elements of $G$ which can be represented as a product of exactly $\ell$ elements from $\ell$ distinct sets from $\mathcal{A}$ taken in any order. That is,
\begin{equation}\label{gen-product-set}
{\Pi}^{\ell}(\mathcal{A}) = \{a_{i_1} \cdots a_{i_{\ell}}: a_{i_j} \in A_{i_j},~\text{where}~ i_j \in [1, m]~\text{for each}~ j \in [1, \ell] ~\text{and}~ i_r \neq i_s ~\text{if}~ r \neq s\}.
\end{equation}
In other words,
\begin{equation*}\label{gen-product-set-1}
{\Pi}^{\ell}(\mathcal{A}) = \bigcup_{\substack{\Lambda \subseteq [1, m]\\ |\Lambda| = \ell}} \left(\bigcup_{\sigma \in S_{\Lambda}} \left(\prod_{i \in \Lambda }A_{\sigma(i)}\right)\right).
\end{equation*}
Clearly, 
\begin{equation}\label{product-set-symmetry}
 {\Pi}^{\ell}(\mathcal{A_{\pi}}) = {\Pi}^{\ell}(\mathcal{A})
\end{equation}
for any permutation $\pi \in S_m$, where $S_m$ denote the group of permutations on the set $\{1, \ldots, m\}$, and $\mathcal{A}_{\pi} = (A_{\pi(1)}, \ldots, A_{\pi(m)})$. 

The generalized product set ${\Pi}^{\ell}(\mathcal{A})$ has connection with product-one subsequence problems. Let $\mathbf{a} = (a_1, \ldots, a_m)$ be a sequence of elements of $G$ (not necessarily abelian), where $m$ is a positive integer. Given a positive integer $\ell$ with $\ell \leq m$, we define
\begin{equation}\label{gen-subseq-sum-eq1}
  \Pi^{\ell} (\mathbf{a}) = \{a_{i_1} \cdots a_{i_{\ell}}: i_j \in [1, m]~ \text{for}~ j = 1, \ldots, \ell\}.
\end{equation}
Thus $\Pi^{\ell} (\mathbf{a})$ is the set of all group elements which can be represented as a product of $\ell$ terms of the sequence $\mathbf{a}$ taken in any order. For the given sequence $\mathbf{a}$, if we define $\mathcal{A} = (A_1, \ldots, A_m)$, where $A_i = \{a_i\}$ for each $i \in [1, m]$. Then 
\[\Pi^{\ell} (\mathbf{a}) = \Pi^{\ell} (\mathcal{A}).\]
If the group $G$ is written additively, then we write $\Sigma^{\ell} (\mathbf{a})$ instead of $\Pi^{\ell} (\mathbf{a})$, and this definition coincides with the definition of $\Sigma^{\ell} (\mathbf{a})$ given in \cite{devos2009} in case the group $G$ is abelian also. The sequeunce sums $\Sigma^{\ell} (\mathbf{a})$ has been studied by various researchers in abelian groups (see, for example, \cite{devos2009} and the references given therein).

For a finite (not necessarily abelian) group $G$, the {\it invariant $d(G)$} is defined as the smallest positive integer $t$ such that every sequence $\mathbf{a}=(a_1,\ldots,a_m)$ over $G$ with $m \geq t$, there exists a positive integer $\ell \leq m$ such that $1 \in \Pi^{\ell} (\mathbf{a})$. Thus a deeper understanding of $\Pi^{\ell} (\mathbf{a})$ may be useful in the study of problems concerning to subsequence products (or subsequence sumswhen $G$ is written additively). For some results in this direction, the reader may refer the papers of Gao and Thangadurai \cite{gao-thanga2006}, and Grynkiewicz \cite{gryn2013}.

We remark that in contrast to the definition of ${\Sigma}^{\ell}(\mathcal{A})$ in \eqref{gen-sumset}, in the definition \ref{gen-product-set} of ${\Pi}^{\ell}(\mathcal{A})$, the subscripts $i_1, \cdots, i_{\ell}$ are not necessarily taken in increasing order. Therefore, in case of $\ell = m$, while the generalized sumset ${\Sigma}^{\ell}(\mathcal{A})$ reduces to the sumset $A_1 + \cdots + A_{\ell}$, the generalized product set ${\Pi}^{\ell}(\mathcal{A})$ does not reduce to the product set $A_1 \cdots A_{\ell}$ if $G$ is not abelian. However, this relaxation enables us to solve the inverse problems in torsion-free groups which are not necessarily abelian. In fact, if $G$ is an abelian group, the two definitions are the same.

In this paper, we obtain an optimal lower bound for the size of generalized product set ${\Pi}^{\ell}(\mathcal{A})$ in torsion-free group (not necessarily abelian), and characterize the structure of extremal sets, thereby solving the inverse problem. Furthermore, we derive inverse theorems in cyclic groups of prime order by slightly modifying the argument in the proofs for torsion-free groups. Our proof of this result also yields a new proof of DeVos-Goddyn-Mohar Theorem in $\mathbb{Z}_p$. Moreover, we extend these results to an arbitrary abelian group $G$. We also discuss some applications of our results to subsequence sums $\Pi^{\ell} (\mathbf{a})$. One of the main results is the following theorem which establishes  an optimal lower bound on the size of the generalized product set $\Pi^{\ell}(\mathcal{A})$ in a torsion-free group $G$.

\begin{theorem} \label{productset-tf-dir-thm}
	Let $\ell$ and $m$ be positive integers such that $\ell \leq m$. Let $\mathcal{A} = (A_1, \ldots, A_m)$ be a sequence of finite subsets of torsion-free group $G$, and let $A = A_1 \cup \cdots \cup A_m$. Let
	\[\mu_{\mathcal{A}} (a) = \min \bigg (\ell, \sum_{j = 1}^m \chi_{A_j}(a) \bigg),\]
	where $a \in A$. Then
	\begin{equation}\label{productset-tf-dir-thm-eq1}
		|\Pi^{\ell}(\mathcal{A})| \geq \sum_{a \in A} \mu_{\mathcal{A}} (a) - \ell + 1.
	\end{equation}
	The lower bond in \eqref{productset-tf-dir-thm-eq1} is best possible.
\end{theorem}

The following corollary follows immediately from the above theorem.
\begin{corollary}\label{productset-tf-dir-thm-cor}
	Let $\ell$ and $m$ be positive integers such that $2 \leq \ell \leq m$. Let $\mathcal{A} = (A_1, \ldots, A_m)$ be a finite sequence of nonempty finite subsets of a torsion-free abelian group $G$, and let $A = A_1 \cup \cdots \cup A_m$. Let $\mu_{\mathcal{A}} (a)$ be defined as in Theorem \ref{productset-tf-dir-thm}. Then
	\begin{equation}\label{productset-tf-dir-thm-cor-eq1}
		|\Pi^{\ell}(\mathcal{A})| \geq \sum_{a \in A} \mu_{\mathcal{A}} (a) - \ell + 1.
	\end{equation}
	The lower bond in \eqref{productset-tf-dir-thm-cor-eq1} is best possible.
\end{corollary}

A nonempty subet $A$ of $G$ is called a {\it geometric progression with common ratio $g$} if it can be expressed as $A = \{a, ag, ag^2, \ldots, ag^{|A| -1}\}$ for some $a, g \in G$ with $g \neq 1$. Unlike Theorem \ref{vosperthm}, if the sets $A_1, \ldots, A_m$ are the extremal sets for the generalized sumset $\Sigma^{\ell}(\mathcal{A})$ (or, the generalized product set $\Pi^{\ell}(\mathcal{A})$), then they need not be arithmetic progressions (or, geometric progressions) as Example \ref{ex-1} given after Corollary \ref{productset-tf-inv-thm1-cor} shows. We give the precise characterization of the extremal sets $A_1, \ldots, A_m$. We show that if the sets $A_1, \ldots, A_m$ are extremal, then they must satisfy some specific properties. To state the result precisely, we need some more definitions. 

\begin{definition}[Geometric progression of type $(a, g, b)$] \label{gp-characterization}
 A nonempty finite subset $A$ of a multiplicative group $G$ is called a {\it geometric progression of type $(a, g, b)$}  if  
\[A = \{ab, agb, \ldots, ag^{|A| - 1}b\}\]
for some $a, b, g \in G$, $g \neq 1$, where $1$ is the identity element of $G$.
\end{definition}

\begin{remark}\label{gp-remark}
Following properties of geometric progression of type $(a, g, b)$ will be useful. 
\begin{enumerate}
\item A geometric progression $\{a, ag^2, \ldots, ag^{k-1}\}$ with common ratio $g$, where $k \geq 2$ is an integer, is clearly a geometric progression of type $(a, g, 1)$.
  \item Let $A$ be a geometric progression of type $(a, g, b)$. Then it is clearly a geometric progression with first term $ab$ and common ratio $b^{-1}gb$, and so it is a geometric progression of type $(ab, b^{-1}gb, 1)$ also. Furthermore, it is easy to see that $A$ can be represented as a geometric progression of type $(1, aga^{- 1}, ab)$ also. If the group $G$ is abelian, then the geometric progression of type $(a, g, b)$ is clearly a geometric progression with first term $ab$ and common ratio $g$.
   
  \item If $A$ is a geometric progression of type $(a, g^{- 1}, b)$, then it is also a geometric progression of type $(a, g, g^{- (|A| - 1)}b)$.
  \item If $A$ is a geometric progression of type $(a, g, b)$, then it is also a geometric progression of type $(ac, c^{-1}gc, c^{-1}b)$ for every $c \in G$. Thus a geometric progression of type $(a, g, b)$ can be represented in many ways, where $g$ is replaced with its conjugate. But if the group $G$ is abelian, then all such representations give the geometric progression with the same common ratio $g$.
\end{enumerate}
\end{remark}

\begin{remark}
  If the group $G$ is written additively, then we will use the term {\it arithmetic progression of type $(a, g, b)$} instead of geometric progression of type $(a, g, b)$. 
\end{remark}

\begin{definition}[$(\ell, g)$- minimizing sequence of sets] \label{min-seq}
	Let $\ell$ and $m$ be positive integers with $\ell \leq m$, and let $\mathcal{A} = (A_1, \dots, A_m)$ be a finite sequence of finite subsets of a multiplicative group $G$. Let $A = A_1 \cup \cdots \cup A_m$. For each $a \in A$, let
\begin{equation}\label{min-seq-def-eq}
  \mu_{\mathcal{A}} (a) = \min \bigg (\ell, \sum_{j = 1}^m \chi_{A_j}(a) \bigg).
\end{equation}
	Then the sequence $\mathcal{A}$ is called an {\it $(\ell, g)$-minimizing sequence of sets} if there exist sets $B_1, \ldots, B_{\ell} \subseteq A$ satisfying the following conditions:
	\begin{enumerate}
		\item $|B_i| \geq 2$ for $i = 1, \ldots, \ell$.
		\item There exists $g \in G$ such that $g \neq 1$ and for each $i = 1, \ldots, \ell$, the set $B_i$ is a geometric progressions of type $(\alpha_i, g, \beta_i)$ for some $\alpha_i, \beta_i \in G$, where $\alpha_{i + 1} = \beta_i^{- 1}$ for each $i \in [1, \ell - 1]$.
		\item For each $a \in A$ we have
		\[\mu_{\mathcal{A}} (a) = \sum_{j = 1}^{\ell} \chi_{B_j}(a).\]
		\item $B_1\cdots B_{\ell} = \Pi^{\ell}(\mathcal{A})$.
	\end{enumerate}	
\end{definition}

\begin{remark} \label{rem-min-seq}
Let $\mathcal{A} = (A_1, \dots, A_m)$ be a $(\ell, g)$-minimizing sequence of sets. Then there exist sets $B_1, \ldots, B_{\ell} \subseteq A$ satisfying the conditions of Definition \ref{min-seq}. The following observation will be useful in our proofs.
\begin{enumerate}
	\item It is easy to show that $A = B_1 \cup	\cdots \cup B_{\ell}$, where $A = A_1 \cup \cdots \cup A_m$.
    \item We have
    \[\sum_{a \in A} \mu_{\mathcal{A}} (a) = \sum_{a \in A} \sum_{j = 1}^{\ell} \chi_{B_j}(a)
	  = \sum_{j = 1}^{\ell} \sum_{a \in A} \chi_{B_j}(a) = \sum_{j = 1}^{\ell} \sum_{a \in B_1 \cup \cdots \cup B_{\ell}} \chi_{B_j}(a) = \sum_{j = 1}^{\ell} |B_j|.\]

	\item If $\sum_{j = 1}^m \chi_{A_j}(a) \leq \ell$ for all $a \in A$, then we have
	\begin{align*}
		\sum_{a \in A} \mu_{\mathcal{A}} (a) = \sum_{a \in A} \min \bigg (\ell, \sum_{j = 1}^m \chi_{A_j}(a) \bigg) &= \sum_{a \in A} \sum_{j = 1}^m \chi_{A_j}(a)\\
       &= \sum_{j = 1}^m \sum_{a \in A} \chi_{A_j}(a) = \sum_{j = 1}^m |A_j|. 
	\end{align*}
	\item  Clearly, $B_1 \cdots B_{\ell} = \{\alpha_1g^j\beta_{\ell} : j \in [0, |B_1| + \cdots |B_{\ell}| - \ell + 1]\}$, and so
	\[|B_1 \cdots B_{\ell}| = |B_1| + \cdots |B_{\ell}| - \ell + 1.\]
\end{enumerate}
\end{remark}

\begin{remark}
If the group $G$ is abelian, then in view of Remark \ref{gp-characterization}, the term ``geometric progressions of type $(\alpha_i, g, \beta_i)$" in the second condition of Remark \ref{gp-remark} can be replaced with the term ``arithmetic progression." The last condition is also written as  $B_1 + \cdots + B_{\ell} = \Sigma^{\ell}(\mathcal{A})$ in this case.
\end{remark}

The next theorem characterizes the sets $A_1, \ldots, A_m$ for which the equality holds in \eqref{productset-tf-dir-thm-eq1}. 

\begin{theorem}\label{productset-tf-inv-thm1}
	Let $\ell$ and $m$ be positive integers such that $2 \leq \ell \leq m$. Let $\mathcal{A} = (A_1, \ldots, A_m)$ be a sequence of finite subsets of torsion-free group $G$, where $|A_i|  \geq 2$ for each $i  \in [1, m]$. Let $A = A_1 \cup \cdots \cup A_m$. Let $\mu_{\mathcal{A}}(a)$ be defined as in \eqref{min-seq-def-eq}. Then
	\begin{equation}\label{productset-tf-inv-thm1-eq1}
		|\Pi^{\ell}(\mathcal{A})| =  \sum_{a \in A} \mu_{\mathcal{A}}(a) - \ell + 1,
	\end{equation}
	if and only if $\mathcal{A}$ is a $(\ell, g)$-minimizing sequence of sets for some $g \in G$. Moreover, if \eqref{productset-tf-inv-thm1-eq1} holds, then the sets $A_1 \cup M, \ldots, A_m \cup M$ are geometric progressions of type $(\alpha_1, g, \beta_1), \ldots, (\alpha_m, g, \beta_m)$, respectively for some $\alpha_i, \beta_i, g \in G$ for $i = 1, \ldots, m$ with $g \neq 1$, where $M = \{a \in A : \mu_{\mathcal{A}}(a) = \ell\}$.
\end{theorem}

If the group $G$ is a torsion-free abelian group, then the set $A$ is also a geometric progression with the same common ratio. The precise result is the following corollary.
\begin{corollary}\label{productset-tf-inv-thm1-cor}
	Let $\ell$ and $m$ be positive integers such that $2 \leq \ell < m$. Let $\mathcal{A} = (A_1, \ldots, A_m)$ be a finite sequence of nonempty finite subsets of a torsion-free abelian group $G$, where $|A_i|  \geq 2$ for each $i  \in [1, m]$. Let $A = A_1 \cup \cdots \cup A_m$, and let $M = \{a \in A : \mu_{\mathcal{A}}(a) = \ell\}$, where $\mu_{\mathcal{A}}(a)$ is defined as in \eqref{min-seq-def-eq}. Then
	\begin{equation}\label{productset-tf-inv-thm1-cor-eq1}
		|\Pi^{\ell}(\mathcal{A})| =  \sum_{a \in A} \mu_{\mathcal{A}}(a) - \ell + 1,
	\end{equation}
	if and only if $\mathcal{A}$ is a $(\ell, d)$-minimizing sequence of sets. Furthermore, if \eqref{productset-tf-inv-thm1-cor-eq1} holds, then the following conclusions hold:
\begin{enumerate}
  \item The sets $A_1 \cup M, \ldots, A_m \cup M$ are geometric progressions with same common ratio $g$ for some $g \in G$ with $g \neq 1$.
  \item The set $A$ is also a geometric progressions with the same common ratio $g$.
\end{enumerate} 
\end{corollary}

The following corollary provides the lower bound on the size of $\Pi^{\ell}(\mathcal{A})$ if $\mu_{\mathcal{A}}(a) < \ell$ for each $a \in A$. The proof easily follows from above corollary and Remark \ref{rem-min-seq}(3).

\begin{corollary}\label{productset-tf-inv-thm1-cor1}
	Let $\ell$ and $m$ be positive integers such that $2 \leq \ell < m$. Let $\mathcal{A} = (A_1, \ldots, A_m)$ be a finite sequence of nonempty finite subsets of a torsion-free abelian group $G$, where $|A_i|  \geq 2$ for each $i  \in [1, m]$. Let $A = A_1 \cup \cdots \cup A_m$, and let $\mu_{\mathcal{A}}(a) < \ell$ for each $a \in A$, where $\mu_{\mathcal{A}}(a)$ is defined as in \eqref{min-seq-def-eq}. Then
	\begin{equation}\label{productset-tf-inv-thm1-cor1-eq1}
		|\Pi^{\ell}(\mathcal{A})| =  \sum_{i = 1}^{m}|A_i| - \ell + 1,
	\end{equation}
	if and only if $\mathcal{A}$ is a $(\ell, d)$-minimizing sequence of sets. Furthermore, if \eqref{productset-tf-inv-thm1-cor1-eq1} holds, then the following conclusions hold:
\begin{enumerate}
  \item The sets $A_1, \ldots, A_m$ are geometric progressions with same common ratio $g$ for some $g \in G$ with $g \neq 1$.
  \item The set $A$ is also a geometric progressions with the same common ratio $g$.
\end{enumerate} 
\end{corollary}

The following example shows that even if the sets $A_1 \cup M, ..., A_m \cup M$ are all geometric progressions with same common ratio, they may not satisfy the equality in \eqref{productset-tf-inv-thm1-cor-eq1}. We present an example in the torsion-free abelian group $G$, written multiplicatively.

\begin{example}\label{ex-1}
Let $\ell = 3$, $m = 5$, and let $g$ be an element of $G$ such that $g \neq 1$. Let $\mathcal{A} = (A_1, \ldots, A_5)$, where $A_1 = \{g^i : i \in [0, 3]\}$, $A_2 = \{g^i : i \in [6, 9]\}$, $A_3 = \{g^i : i \in [7, 10]\}$, $A_4 = \{g^i : i \in [8, 11]\}$, and $A_5 = \{g^i : i \in [9, 12]\}$ are subsets of $G$. Then 
	\[A = A_1 \cup \cdots \cup A_5 = \{g^i : i \in [0, 3] \cup [6, 12]\},\]
	and
	\[M = \emptyset.\]
Hence $A_1 \cup M, \ldots, A_5 \cup M$ are geometric progressions with same common ratio $g$. It is easy to see that $\mu_{\mathcal{A}}(g^i) = 1$ for each $i \in \{0, 1, 2, 3, 6, 12\}$ and $\mu_{\mathcal{A}}(g^i) = 2$ for each $i \in [7, 11]$. Therefore,
	\[\sum_{a \in A} \mu_{\mathcal{A}} (a) = 6 + 10 = 16,\]
	and so 
	\[\sum_{a \in A} \mu_{\mathcal{A}} (a) - \ell + 1 = 14.\]
Since $A_1 A_2 A_3 = \{g^i : i \in [13, 22]\}$ and $A_3 A_4 A_5 = \{g^i : i \in [22, 33]\}$, it follows that
	\[|\Pi^{3}(\mathcal{A})| \geq |(A_1 A_2 A_3) \cup (A_3 A_4 A_5)| = 21 > 14 = \sum_{a \in A} \mu_{\mathcal{A}} (a) - 3 + 1.\]
\end{example}

If $G$ is a torsion-free abelian group and $\mathcal{A} = (A_1, \ldots, A_m)$ is a sequence of nonempty finite subsets of $G$ such that $|\{g \in A : \mu (g) \geq 2 \}| \leq 1$, then we prove that the equality \eqref{productset-tf-inv-thm1-cor-eq1} does not hold in Corollary \ref{productset-tf-inv-thm1-cor}. More precisely, we prove the following theorem.

\begin{theorem}\label{productset-no-min-sq-ext-thm}
	Let $\ell$ and $m$ be positive integer such that $2 \leq \ell < m$. Then no sequence $\mathcal{A} = (A_1, \ldots, A_m)$ of finite subsets of torsion-free abelian group $G$ exists satisfying the conditions $|A_i|  \geq 2$ for each $i  \in [1, m]$, $|\{a \in A : \mu_{\mathcal{A}}(a) \geq 2 \}| \leq 1$, and 
	\[|\Pi^{\ell}(\mathcal{A})| =  \sum_{a \in A} \mu_{\mathcal{A}}(a) - \ell + 1,\]
	where $A = A_1 \cup \cdots \cup A_m$ and $\mu_{\mathcal{A}}(a)$ is defined as in \eqref{min-seq-def-eq}.
\end{theorem}

The next theorem states that if the sets $A_1, \ldots, A_m$ are extremal sets for the generalized product set $\Pi^{\ell}(\mathcal{A})$, then certain subsets of the set $A = A_1 \cup \cdots \cup A_m$  have a specific structure. The precise result is stated in the following theorem. 

\begin{theorem}\label{productset-tf-inv-thm2}
	Let $\ell$ and $m$ be positive integers such that $2 \leq \ell < m$. Let $\mathcal{A} = (A_1, \ldots, A_m)$ be a sequence of finite subsets of torsion-free group $G$ such that $|A_i|  \geq 2$ for each $i  \in [1, m]$, and let $A = A_1 \cup \cdots \cup A_m$. Let $\mu_{\mathcal{A}}(a) = \ell$ for some $a \in A$, where $\mu_{\mathcal{A}}(a)$ is defined as in \eqref{min-seq-def-eq}. Furthermore, assume that $|\{a \in A : \mu_{\mathcal{A}}(a) \geq 2\}| \geq 2$. Let
\begin{equation}\label{productset-tf-inv-thm2-eq1}
  |\Pi^{\ell}(\mathcal{A})| =  \sum_{a \in A} \mu_{\mathcal{A}}(a) - \ell + 1,
\end{equation}
and
		\[A_j' = \{a \in A : \mu_{\mathcal{A}}(a) \geq j\}\]
		for each $j \in [1, \ell]$, and let $k$ be the largest integer such that $|A_j'| \geq 2$ for each $j \in [1, k]$. Then  for each $j \in [1, k]$, the set $A_j'$ is a geometric progression of type $(\alpha_j, g, \beta_j)$ for some $\alpha_j, \beta_j, g \in G$  with $g \neq 1$, where $\alpha_{j + 1} = \beta_j^{- 1}$ for each $j \in [1, k - 1]$. 
\end{theorem}

\textbf{Organization of the paper:} In Section $2$, we fix some notations which will be used throughout the paper. In Section \ref{sec-torsion-free-groups}, we prove the main results. This section is divided into three subsections. In Subsection \ref{subsec-gp-properties}, we prove some auxiliary lemmas related to the properties of geometric progressions in groups, and in Subsection \ref{subsec-aux-lem}, we prove additional auxiliary lemmas required for the proof of main results. In Subsection \ref{subsec-main-thm-torsionfree}, we prove the main results in torsion-free groups: Theorem \ref{productset-tf-dir-thm}, Theorem \ref{productset-tf-inv-thm1}, Corollary \ref{productset-tf-inv-thm1-cor}, Theorem \ref{productset-no-min-sq-ext-thm} and Theorem \ref{productset-tf-inv-thm2}. In Section \ref{sec-prime-order-group}, we prove analogous results in the cyclic groups of prime orders. The main results proved in this section are Theorem \ref{zp-dir-thm} and Theorem \ref{zp-inv-thm}. In Section \ref{sec-abelian-groups}, we prove analogous results in arbitrary abelian groups. The main results proved in this section are Theorem \ref{sigma-abl-dir-thm} and Theorem \ref{sigma-abl-inv-thm}. In Section \ref{applications-to-subseq-sum}, we discuss some applications of these results to obtain some new results on subsequence sums $\Pi^{\ell} (\mathbf{a})$ (see Theorem \ref{gen-subseq-sum-tf-dir-thm1}, Theorem \ref{gen-subseq-sum-gp-dir-thm1}, Theorem \ref{gen-subseq-sum-tf-dir-thm2}, Theorem \ref{gen-subseq-sum-gp-dir-thm2}, Theorem \ref{gen-subseq-sum-tf-inv-thm1} and Theorem \ref{gen-subseq-sum-tf-inv-thm2}).   

The proof of Theorem \ref{productset-tf-dir-thm} and Theorem \ref{productset-tf-inv-thm1} will require various auxiliary lemmas including the multiple set version (see Lemma \ref{kemp-gen-thm-tf} in Section \ref{sec-torsion-free-groups}) of the following theorem of Kemperman \cite{kemperman1956}.
\begin{theorem}\label{kemp-thm-tf}
    	Let $G$ be a torsion-free multiplicative group. Let $A$ and $B$ be nonempty finite subsets of $G$. Then
    \[|A B| \geq |A| + |B| - 1.\]	
\end{theorem}

The characterization of a geometric progression given in Definition \ref{gp-characterization} helps us to generalize the following theorem of Brailovsky and Freiman \cite{brailovsky1990} for the sumset of multiple sets (see Lemma \ref{brail-gen-inv-thm-tf} in Section \ref{sec-torsion-free-groups}) which will be required in the proof Theorem \ref{productset-tf-inv-thm1} and Theorem \ref{productset-tf-inv-thm2}.

\begin{theorem}\label{brail-inv-thm-tf}
	Let $G$ be a torsion-free group. Let $A$ and $B$ be nonempty finite subsets of $G$ such that $|A|, |B| \geq 2$ and 
	\[|A B| = |A| + |B| - 1.\]
	Then $A$ and $B$ are geometric progressions of types $(a, g, 1)$ and $(1, g, b)$, respectively, where $a, b, g \in G$, $g \neq 1$.
\end{theorem}

\section{General notation}
 Let $\mathcal{A} = (A_1, \dots, A_m)$ be a finite sequence of finite subsets of $G$, and let $A = A_1 \cup \cdots \cup A_m$. Then
 \begin{enumerate}
 	\item Recall that for each $a \in A$, we define
 	\[\mu_{\mathcal{A}}(a) = \min \bigg (\ell, \sum_{j = 1}^m \chi_{A_j}(a) \bigg).\]
 	\item For each $a \in A$, we define
 	\[\eta_{\mathcal{A}} (a) = \sum_{j = 1}^{\ell} \chi_{A_j}(a).\]
 	\item For each $a \in A$,
 	\[\tau_{\mathcal{A}} (a) = \min \biggl(\ell, \sum_{j = \ell + 1}^m \chi_{A_j}(a)\biggr).\]
    \item Clearly, if $\mu_{\mathcal{A}}(a) < \ell$ for some $a \in A$, then 
 \[\mu_{\mathcal{A}}(a) = \eta_{\mathcal{A}} (a) + \tau_{\mathcal{A}} (a).\]
 \end{enumerate}	
From now on, if the sequence $\mathcal{A}$ is clear from the context, we will drop the subscript $\mathcal{A}$ from the above expressions used in the proofs and simply write $\mu(a)$, $\eta(a)$ and $\tau(a)$ in place of $\mu_{\mathcal{A}}(a)$, $\eta_{\mathcal{A}} (a)$ and $\tau_{\mathcal{A}} (a)$, respectively.
 
\section{Generalized product set in torsion-free groups}\label{sec-torsion-free-groups}
In the next subsection, we prove some lemmas which describes the geometric progressions in a group (not necessarily abelian). In Subsection \ref{subsec-aux-lem}, we prove various auxiliary lemmas required for the proof of Theorem \ref{productset-tf-dir-thm}, Theorem \ref{productset-tf-inv-thm1}, Corollary \ref{productset-tf-inv-thm1-cor}, Theorem \ref{productset-no-min-sq-ext-thm} and Theorem \ref{productset-tf-inv-thm2}. The case $\ell = 2$ of Theorem \ref{productset-tf-inv-thm1} is separately proved as Lemma \ref{aux-lem-special-case}. In Subsection \ref{subsec-main-thm-torsionfree}, we prove the main results in torsion-free groups.

\subsection{Properties of geometric progressions in groups} \label{subsec-gp-properties}

\begin{lemma}\label{gp-lem1}
	Let $A$ be a nonempty finite subset of torsion-free group $G$ such that $|A| = m \geq 2$. Let $A$ be a geometric progression of type $(a, g, b)$ and also a geometric progression of type $(\alpha, g_1, \beta)$ for some $a, b, \alpha, \beta, g, g_1 \in G$ with $g \neq 1, g_1 \neq 1$. Then one of the following conditions holds:
	\begin{enumerate}
		\item $a = \alpha c$, $b = c^{- 1} \beta$, and $g = c^{- 1} g_1 c$ for some $c \in G$.
		\item $a = \alpha c$, $b = c^{- 1} g_1^{m - 1}\beta$, and $g = c^{- 1} g_1^{- 1} c$ for some $c \in G$.
	\end{enumerate}
\end{lemma}

\begin{proof} Since $ab \in A$ and $A$ is a geometric progression of type $(\alpha, g_1, \beta)$, we may assume that $ab = \alpha g_1^r \beta$ for some $r \in [0, m - 1]$. Hence
	\[\alpha^{- 1} a = g_1^r \beta b^{- 1}.\]
	Let $c = \alpha^{- 1} a = g_1^r \beta b^{- 1}$. Then
	\begin{equation}\label{gp-lem1-eq1}	
		a = \alpha c ~\text{and}~ b = c^{- 1} g_1^r \beta.
	\end{equation}
	Since $a g b \in A = \{\alpha \beta, \alpha g_1 \beta, \ldots, \alpha g_1^{m - 1} \beta\}$, it follows that there exists $s \in [0, m - 1 ] \setminus \{r\}$ such that
	\[a g b = \alpha g_1^s \beta.\]
It follows from \eqref{gp-lem1-eq1} that
	\[\alpha c g c^{- 1} g_1^r \beta = \alpha g_1^s \beta,\]
	and so
	\begin{equation}\label{gp-lem1-eq2}
		g = c^{- 1} g_1^{s - r} c.
	\end{equation}
Hence it follows \eqref{gp-lem1-eq1} and \eqref{gp-lem1-eq2} that
\begin{equation}\label{gp-lem1-eq3}
	a g^j b = a c^{- 1} g_1^{j(s -r)} c b = \alpha g_1^{js - (j - 1)r} \beta
\end{equation}
	for each $j \in [1, m - 1]$. 
	
	\noindent {\textbf{Claim 1}} (Either $r = 0$ or $r = m-1$). For $m = 2$, it is trivial. Now assume that $m \geq 3$. If $r \in [1, m - 2]$, then it follows that there exist distinct $i_1, i_2 \in [0, m - 1]$ such that
	\[\alpha g_1^{r - 1}\beta = ag^{i_1}b,\]
	and
	\[\alpha g_1^{r + 1}\beta = ag^{i_2}b.\]
	Hence it follows from \eqref{gp-lem1-eq3} that
	\[\alpha g_1^{i_1(s - r) + r} \beta = \alpha g_1^{r - 1}\beta\]
	and 
	\[\alpha g_1^{i_2(s - r) + r} \beta = \alpha g_1^{r + 1}\beta.\]
	Therefore,
	\[g_1^{i_1(s - r) + 1} = 1 ~\text{and}~ g_1^{i_2(s - r) - 1} = 1.\]
	Since $g_1 \neq 1$ and $G$ is a torsion-free group, it follows that
	\[i_1(s - r) + 1 = 0 ~\text{and}~ i_2(s - r) - 1 = 0,\]
	and so
	\[(i_1 + i_2) (s - r) = 0\]
	which is not possible because $i_1, i_2, s, r \in [0, m - 1]$ such that $i_1 \neq i_2$ and $s \neq r$. Hence either $r = 0$ or $r = m - 1$. We consider each of these cases below.
	
	\noindent {\textbf{Case 1}} ($r = 0$). In this case, we show that $s = 1$. For $m = 2$, it is trivial. Now, we assume that $m \geq 3$. Suppose that $s \in [2, m - 1]$. Then it follows that there exists $t \in [1, m - 1]$ that 
	\[ag^tb = \alpha g_1 \beta.\]
	Hence it follows from \eqref{gp-lem1-eq3} that
	\[g_1^{ts - 1} = 1.\]
	Since $G$ is a torsion-free group, it follows that $ts = 1$, which is a contradiction. Hence we must have $s = 1$. Thus if $r = 0$, then $s = 1$, and so the first condition of the lemma holds.
	
	\noindent {\textbf{Case 2}} ($r = m - 1$). In this case, we show that $s = m - 2$. For $m = 2$, it is trivial. Now, we assume that $m \geq 3$. 
	Suppose that $s \in [0, m - 3]$. Then it follows that there exists $t \in [0, m - 2]$ such that 
	\[ag^tb = \alpha g_1^{m - 2} \beta.\]
	Hence it follows from \eqref{gp-lem1-eq3}  that
	\[g_1^{t(s - m + 1) + 1} = 1.\]
	Since $g_1 \neq 1$ and $G$ is a torsion-free group, it follows that $t(m - 1 - s) = 1$, which is a contradiction because $s \in [0, m - 3]$. Hence we must have $s = m-2$. Thus if $r = m-1$, then $s = m-2$, and so the second condition of the lemma holds. This completes the proof.
\end{proof}

\begin{lemma}\label{gp-lem2}
	Let $A$ and $B$ be nonempty finite subsets of torsion-free group $G$ such that $|A|, |B| \geq 2$. Let $A$ and $B$ are geometric progression of types $(\alpha, g_1, \beta)$ and $(p, g_1, q)$, respectively for some $\alpha, \beta, p, q, g \in G$ with $g \neq 1$. If $A$ can be represented as a geometric progression of type $(a, g, b)$, then $B$ can be represented as a geometric progression of type $(\gamma, g, \delta)$ for some $\gamma, \delta \in G$.
\end{lemma}

\begin{proof}
	Since $A = \{ab, a g b, \ldots, a g^{|A| - 1} b\} = \{\alpha\beta, \alpha g_1\beta, \ldots, \alpha g_1^{|A| - 1}\beta\}$, it follows from Lemma \ref{gp-lem1} that either $g_1 = cgc^{- 1}$ or $g_1 = cg^{- 1}c^{- 1}$ for some $c \in G$. If $g_1 = cgc^{- 1}$, then $B$ is a geometric progression of type $(p, cgc^{- 1}, q)$, and so it is a geometric progression of type $(pc, g, c^{- 1}q)$. If $g_1 = cg^{- 1}c^{- 1}$, then $B$ is a geometric progression of type $(p, cg^{- 1}c^{- 1}, q)$, and so it is a geometric progression of type $(pc g^{|B| - 1}, g, c^{- 1}q)$. Thus in any case, $B$ is a geometric progression of type $(\gamma, g, \delta)$ for some $\gamma, \delta \in G$. This completes the proof.
\end{proof}

\begin{lemma}\label{gp-lem4}
	Let $A$ and $B$ be finite subsets of torsion-free group $G$ such that $|A| \geq 2$. Let $A$ and $A \cup B$ be geometric progression. If $A \cup B$ be can be represented as a geometric progression of type $(\alpha, g, \beta)$, then the set $A$ can also be represented as a geometric progression of type $(\alpha, g^p, g^r \beta)$ for some integers $p$ and $r$ such that $p \geq 1$, $r \geq 0$.
\end{lemma}

\begin{proof}
	Clearly, $A \cup B$ be geometric progression of type $(\alpha_1, g_1, 1)$, where $g_1 = \beta^{- 1}g \beta$ and $\alpha_1 = \alpha \beta$. Since $A \subseteq A \cup B$ and $A \cup B = \{\alpha_1 , \alpha_1 g_1 , \ldots, \alpha_1 g_1^{m - 1}\}$, it follows that 
	\[A = \{\alpha_1 g_1^{p_1} , \alpha_1 g_1^{p_2} , \ldots, \alpha_1 g_1^{p_n}\},\]
	where $0 \leq p_1 < p_2 < \cdots < p_n$ with $|A| = n$.
	Since $A$ a geometric progression, it follows that
	\[A = \{\beta_1, \beta_1 h, \ldots, \beta_1 h^{n - 1}\}\]
	for some $\beta_1, h \in G$ with $h \neq 1$. Since $\beta_1, \beta_1 h \in A = \{\alpha_1 g_1^{p_1} , \alpha_1 g_1^{p_2} , \ldots, \alpha_1 g_1^{p_n}\}$, it follows that there exist $u, v \in \{p_1, \ldots, p_n\}$ such that 
	\begin{equation*}
		\beta_1 = \alpha_1 g_1^{p_u},
	\end{equation*}
	and 
	\begin{equation*}
		\beta_1 h = \alpha_1 g_1^{p_v}.
	\end{equation*}
	Hence
	\begin{equation*}
		h = g_1^{p_v - p_u}.
	\end{equation*}
	Therefore,
	\[A = \{\alpha_1 g_1^{p_u} , \alpha_1 g_1^{(p_v - p_u) + p_u}, \ldots, \alpha_1 g_1^{(p_v - p_u) (n - 1) + p_u}\}.\]
	It is easy to see that $(p_v - p_u) i + p_u \geq 0$ for each $i \in [0, n - 1]$. Now, if $p_v - p_u > 0$, then we choose
		\[p = p_v - p_u ~\text{and}~ r = p_u,\]
		and if $p_v - p_u < 0$, then we choose
		\[p = p_u - p_v ~\text{and}~ r = (p_v - p_u) (n - 1) + p_u.\] 
	Therefore,
	\[A = \{\alpha_1 g_1^r, \alpha_1 g_1^{p + r}, \ldots, \alpha_1 g_1^{(n - 1)p + r}\}\]
	for some $p \geq 1$ and $r \geq 0$, and so $A$ is a geometric progression of the type $(\alpha_1g_1^r, g_1^p,  1)$, and so it is of the type $(\alpha, g^p, g^r \beta)$. This completes the proof. 	
\end{proof}

\begin{lemma}\label{gp-lem5}
	Let $G$ be an abelian group. Let $A$ and $B$ be nonempty finite subsets of $G$. Let $A$ and $B$ are geometric progressions with the same common ratio $g \in G$ with $g \neq 1$. Furthermore, assume that $A \cap B \neq \emptyset$. Then $A$, $B$ and $A \cup B$ are geometric progressions with the same common ratio $g$.
\end{lemma}

\begin{proof}
	Let $|A| = m$, $|B| = n$, and let $\alpha \in A \cap B$. Then there exist $a \in A$ and $b \in B$ such that
	\[A = \{a g^j : j \in [0, m - 1]\},\]
	and
	\[B = \{b g^j : j \in [0, n - 1]\}.\]
	Since $\alpha \in A \cap B$, it follows that that there exist nonnegative integers $i_0 \in [0, m - 1]$ and $j_0 \in [0, n - 1]$ such that
	\[\alpha = a g^{i_0} = b g^{j_0}.\]
	Since $\alpha = a g^{i_0} = b g^{j_0}$, it follows that
	\[A = \{\alpha g^j : j \in [- i_0, m - i_0 - 1]\},\]
	and
	\[B = \{\alpha g^j : j \in [- j_0, n - j_0 - 1]\}.\]
	Let
	\[r = \max (i_0, j_0) ~\text{and}~ s = \max (m - i_0 - 1, n - j_0 - 1).\] 
	Then, 
	\[A\cup B = \{\alpha g^j : j \in [- r, s]\}.\] 
	Hence $A \cup B$ is also a geometric progression with  common ratio $g$.	
\end{proof}

\subsection{Auxiliary lemmas} \label{subsec-aux-lem}
We require the following additional lemmas for the proof of main results.

\begin{lemma}\label{kemp-gen-thm-tf}
	Let $\ell \geq 2$ be an integer. Let $G$ be a torsion-free group. Let $A_1, \ldots, A_{\ell}$ be nonempty finite subsets of $G$. Then
	\begin{equation}\label{kemp-gen-thm-tf-eq1}
		|A_1 \cdots A_{\ell}| \geq |A_1| + \cdots + |A_{\ell}| - \ell + 1.
	\end{equation}
	The lower bound in \eqref{kemp-gen-thm-tf-eq1} is best possible.
\end{lemma}

\begin{proof}
	The lower bound in \eqref{kemp-gen-thm-tf-eq1} follows easily from Theorem \ref{kemp-thm-tf} and by applying induction on $\ell$. 

	Next we show that this lower is best possible. Let $n_1, \ldots, n_{\ell}$ be positive integers. Let $A_1, \ldots, A_{\ell}$ be subsets of $G$ such that $|A_i| = n_i$ for each $i \in [1, \ell]$, and let  for each $i \in [1, \ell]$, the set $A_i$ is a geometric progressions of type $(\alpha_i, g, \beta_i)$. Then
	\[A_1 \cdots A_{\ell} = \{\alpha_1g^j\beta_{\ell} : j \in [0, n_1 + \cdots + n_{\ell} - \ell + 1]\},\]
	and so
	\[|A_1 \cdots A_{\ell}| = |A_1| + \cdots + |A_{\ell}| - \ell + 1.\]
	Thus the lower bound in \eqref{kemp-gen-thm-tf-eq1} is best possible. This completes the proof.
\end{proof}

\begin{lemma}\label{aux-lem-inv1}
	Let $G$ be a torsion-free group. Let $A$ and $B$ be nonempty finite subsets of $G$ such that $|A| \geq 2, |B| \geq 2$, and 
	\[|A B| = |A| + |B| - 1.\]
	Then the subsets $A$ and $B$ are geometric progressions of types $(a, g, b)$ and $(b^{- 1}, g, c)$, respectively, where $a, b, c, g \in G$, $g \neq 1$.
\end{lemma}

\begin{proof}
	Since $|A B| = |A| + |B| - 1$, it follows from Theorem \ref{brail-inv-thm-tf} that the subsets $A$ and $B$ are geometric progressions of types $(a, g, 1)$ and $(1, g, b)$, respectively. Hence $A$ and $B$ are geometric progression of types $(ab^{- 1}, bgb^{- 1}, b)$ and $(b^{- 1}, bgb^{- 1}, b^2)$, respectively. This completes the proof.
\end{proof}

\begin{remark}
	Lemma \ref{aux-lem-inv1} and Theorem \ref{brail-inv-thm-tf} are equivalent.
\end{remark}

\begin{lemma}\label{aux-lem-inv2}
	Let $A$ and $B$ be nonempty finite subsets of torsion-free group $G$ such that $|A|, |B| \geq 2$ and $|A B| = |A| + |B| - 1$. Let $A$ be a geometric progression of type $(a, g, b)$. Then $B$ is a geometric progression of type $(b^{- 1}, g, \delta)$ for some $\delta \in G$.
\end{lemma}

\begin{proof}
	Since $|A B| = |A| + |B| - 1$, it follows from Lemma \ref{aux-lem-inv1} that 
	\[A = \{\alpha\beta, \alpha g_1\beta, \ldots, \alpha g_1^{|A| - 1}\beta\},\]
	and
	\[B = \{\beta^{- 1}\gamma, \beta^{- 1}g_1\gamma, \ldots, \beta^{- 1}g_1^{|B| - 1}\gamma\}.\]
	for some $\alpha, \beta, \gamma, g \in G$ with $g \neq 1$. Since $A$ is geometric progression of type $(a, g, b)$ and also of type $(\alpha, g_1, \beta)$, it follows from Lemma \ref{gp-lem1} that 
	\[a = \alpha c, b = c^{- 1} \beta, g = c^{- 1} g_1 c\]
	or
	\[a = \alpha c, b = c^{- 1} g_1^{m - 1}\beta, g = c^{- 1} g_1^{- 1} c\]
	for some $c \in G$. In the first case, $B$ is a geometric progression of type $(b^{- 1} c^{- 1}, cgc^{- 1}, \gamma)$, and so it can be represented as the geometric progression of type $(b^{- 1}, g, c^{- 1} \gamma)$. In the second case, $B$ is geometric progression of type $(b^{- 1} g^{- (m - 1)}c^{- 1}, cg^{- 1}c^{- 1}, \gamma)$, and so it can be represented as the geometric progression of type $(b^{- 1}, g, g^{- 2(m - 1)}c^{- 1} \gamma)$. Thus in each case, the set $B$ is geometric progression of type $(b^{- 1}, g, \delta)$ for some $\delta \in G$. This completes the proof.
\end{proof}

\begin{lemma}\label{brail-gen-inv-thm-tf}
	Let $\ell \geq 2$ be an integer, and let $A_1, \ldots, A_{\ell}$ be nonempty finite subsets of a torsion-free group $G$ such that $|A_i| \geq 2$ for each $i \in [1, \ell]$. Let
	\[|A_1 \cdots A_{\ell}| = |A_1| + \cdots + |A_{\ell}| - \ell + 1.\]
	Then for each $i \in [1, \ell - 1]$, the set $A_i$ is a geometric progressions of type $(\alpha_i, g, \beta_i)$ for some $\alpha_i, \beta_i \in G$, where $\alpha_{i + 1} = \beta_i^{- 1}$ and $g \neq 1$.
\end{lemma}

\begin{proof}
	First assume that for each $i \in [1, \ell - 1]$, the set $A_i$ is a geometric progressions of type $(\alpha_i, g, \beta_i)$ for some $\alpha_i, \beta_i \in G$, where $\alpha_{i + 1} = \beta_i^{- 1}$ and $g \neq 1$. In this case, it is easy to verify that $|A_1 \cdots A_{\ell}| = |A_1| + \cdots + |A_{\ell}| - \ell + 1$. Conversely, assume that 
\[|A_1 \cdots A_{\ell}| = |A_1| + \cdots + |A_{\ell}| - \ell + 1.\]
If $\ell = 2$, then the result follows from Lemma \ref{aux-lem-inv1}. Hence assume that $\ell \geq 3$. Suppose that Lemma \ref{brail-gen-inv-thm-tf} is true for $\ell - 1$ sets. Let $A = A_1 \cdots A_{\ell - 1}$. Then it follows from Lemma \ref{kemp-gen-thm-tf} that
	\begin{align*}
		|A_1| + \cdots + |A_{\ell}| - \ell + 1 &= |A_1 \cdots A_{\ell}|\\
		&= |A A_{\ell}| \\
		&\geq |A| + |A_{\ell}| - 1\\
		&= |A_1 \cdots A_{\ell - 1}| + |A_{\ell}| - 1\\
		&\geq (|A_1| + \cdots + |A_{\ell - 1}| - \ell + 2) + |A_{\ell}| - 1\\
		&= |A_1| + \cdots + |A_{\ell}| - \ell + 1.
	\end{align*}
	Hence
	\[|A A_{\ell}| = |A| + |A_{\ell}| - 1,\]
	and
	\[|A_1 \cdots A_{\ell - 1}| = |A_1| + \cdots + |A_{\ell - 1}| - \ell + 2.\]
	Since $|A_1 \cdots A_{\ell - 1}| = |A_1| + \cdots + |A_{\ell - 1}| - \ell + 2$, it follows from the induction hypothesis that for each $i \in [1, \ell - 1]$ the sets $A_i$ is a geometric progressions of type $(\alpha_i, g, \beta_i)$ for some $\alpha_i, \beta_i \in G$, where $\alpha_{i + 1} = \beta_i^{- 1}$ and $g \neq 1$. Therefore,
	\[A_1 = \{\alpha_1\beta_1, \alpha_1g\beta_1, \ldots, \alpha_1g^{|A_1| - 1}\beta_1\}\]
	and 
	\[A_i = \{\beta_{i - 1}^{- 1}\beta_i, \beta_{i - 1}^{- 1}g\beta_i, \ldots, \beta_{i - 1}^{- 1}g^{|A_i| - 1}\beta_i\}\]
	for each $i \in [2, \ell - 1]$. Hence
	\[A = A_1 \cdots A_{\ell - 1} = \{\alpha_1\beta_{\ell - 1}, \alpha_1g\beta_{\ell - 1}, \ldots, \alpha_1g^{(|A_1| + \cdots + |A_{\ell - 1}|- \ell + 1)}\beta_{\ell - 1}\}.\]
	Since $|A A_{\ell}| = |A|+  |A_{\ell}| - 1$ and $A$ is geometric progression of types $(\alpha_1, g, \beta_{\ell - 1})$, it follows from Lemma \ref{aux-lem-inv2} that $A_{\ell - 1}$ is geometric progression of type $(\beta_{\ell - 1}^{- 1}, g, \beta_{\ell})$ for some $\beta_{\ell} \in G$. Therefore, the sets $A_1, \ldots, A_{\ell}$ are geometric progressions of desired types. This completes the proof.
\end{proof}

\begin{lemma}\label{aux-lem-inv3}
Let $\ell$ and $m$ be integers such that $2 \leq \ell \leq  m$. Let $\mathcal{A} = (A_1, \dots, A_m)$ be a finite sequence of finite subsets of a group $G$, and let $A = A_1 \cup \cdots \cup A_m$.	If $\mathcal{A}$ is a $(\ell, g)$-minimizing sequence of sets for some $g \in G$, then
	\[|\Pi^{\ell}(\mathcal{A})| =  \sum_{a \in A} \mu (a) - \ell + 1.\]
\end{lemma}

\begin{proof}
The proof follows from Remark \ref{rem-min-seq}.
\end{proof} 

\begin{lemma}\label{aux-lem1}
	Let $\ell$ and $m$ be integers such that $2 \leq \ell \leq  m$. Let $\mathcal{A} = (A_1, \dots, A_m)$ be a finite sequence of finite subsets of a group $G$. Let $A = A_1 \cup \cdots \cup A_m$, and $M = \{a \in A: \mu(a) = \ell\}$. For each $j \in [1, \ell]$, let 
\begin{equation}\label{aux-lem1-eq1}
  B_j = \{a \in (A_{\ell + 1} \cup \cdots \cup A_m) \setminus M: \tau(a) \geq j\}
\end{equation}
and
\begin{equation}\label{aux-lem1-eq2}
A_j^1 = A_j \cup B_j \cup M.
\end{equation}
Then the following statements hold:
\begin{enumerate}
\item For each $a \in A$, we have
\begin{equation}\label{aux-lem1-eq3}
  \mu(a) \geq \sum_{j = 1}^{\ell} \chi_{A_j^1}(a).
\end{equation}
Moreover, if the strict inequality holds in \eqref{aux-lem1-eq3} for some $a \in A$, then
\begin{equation}\label{aux-lem1-eq4}
  a \in (A_1 \cup \cdots \cup A_{\ell}) \cap (A_{\ell + 1} \cup \cdots \cup A_m)~~ \text{and}~~ a \not\in M.
\end{equation}
\item Let 
	\[X = \biggl\{a \in \biggl((A_1 \cup \cdots \cup A_{\ell}) \cap (A_{\ell + 1} \cup \cdots 
	\cup A_m)\biggr) \setminus M : \sum_{j = 1}^{\ell} \chi_{A_j^1}(a) < \mu(a) \biggr\}.\]
Then for each $a \in X$, we have
\begin{equation}\label{aux-lem1-eq5}
 \mu(a) >  \tau(a) + \sum_{j = \tau(g) + 1}^{\ell} \chi_{A_j}(a)
\end{equation}
\end{enumerate}
\end{lemma}

\begin{proof}
If $a \in M$, then $\mu(a) = \ell \geq \sum_{j = 1}^{\ell} \chi_{A_j^1}(a)$. If $a \not\in M$, then $\mu(a) < \ell$, and so
Therefore,
\begin{equation}\label{aux-lem1-eq8}
  \mu(a) = \sum_{j = 1}^{m} \chi_{A_j}(a)= \sum_{j = 1}^{\ell} \chi_{A_j}(a) + \sum_{j = \ell + 1}^{m} \chi_{A_j}(a).
\end{equation}
Since $a \not \in M$, it follows from \eqref{aux-lem1-eq1} that 
\[\sum_{j = \ell + 1}^{m} \chi_{A_j}(a) \geq \sum_{j = 1}^{\ell} \chi_{B_j}(a) \geq \sum_{j = 1}^{\ell} \chi_{(B_j \cup M)}(a).\]
Therefore, it follows from \eqref{aux-lem1-eq8} that
\begin{align*}
	\mu(a) \geq \sum_{j = 1}^{\ell} \chi_{A_j}(a) + \sum_{j = 1}^{\ell} \chi_{(B_j \cup M)}(a) \geq \sum_{j = 1}^{\ell} \chi_{(A_j \cup B_j \cup M)}(a) = \sum_{j = 1}^{\ell} \chi_{A_j^1}(a),
\end{align*}
which proves \eqref{aux-lem1-eq3}.

Now assume that strict inequality holds in \eqref{aux-lem1-eq3} for some $a \in A$. Then
\begin{equation}\label{aux-lem1-eq9}
\mu(a) > \sum_{j = 1}^{\ell} \chi_{A_j^1}(a),
\end{equation}
which implies that $a \not\in M$, for if $a \in M$, then
	\[\mu(a) = \ell = \sum_{j = 1}^{\ell} \chi_{A_j^1}(a),\]
which contradicts \eqref{aux-lem1-eq9}. Since $a \not \in M$, we have $\mu(a) < \ell$, and $\mu(a) = \tau(a) + \eta(a)$. Now we show that $a \in (A_1 \cup \cdots \cup A_{\ell}) \cap (A_{\ell + 1} \cup \cdots \cup A_m)$. Suppose that
	$a \notin (A_1 \cup \cdots \cup A_{\ell}) \cap (A_{\ell + 1} \cup \cdots \cup A_m)$. Then either
	\[a \in (A_1 \cup \cdots \cup A_{\ell}) \setminus (A_{\ell + 1} \cup \cdots \cup A_m \cup M)\]
	or
	\[a \in (A_{\ell + 1} \cup \cdots \cup A_m) \setminus (A_1 \cup \cdots \cup A_{\ell} \cup M).\]
	If $a \in (A_1 \cup \cdots \cup A_{\ell}) \setminus (A_{\ell + 1} \cup \cdots \cup A_m \cup M)$, then by the definition of $\tau(a)$, it follows that $\tau(a) = 0$, and so

	\begin{equation}\label{aux-lem1-eq10}
		\mu (a) = \eta(a) = \sum_{j = 1}^{\ell} \chi_{A_j}(a).
	\end{equation}
	Since $A_j \subseteq A_j^0 \subseteq A_j^1$, it follows from \eqref{aux-lem1-eq10} that
	\[\mu (a) =  \sum_{j = 1}^{\ell} \chi_{A_j}(g) \leq \sum_{j = 1}^{\ell} \chi_{A_j^1}(a),\]
	which is a contradiction. Therefore,
	\[a \in (A_{\ell + 1} \cup \cdots \cup A_m) \setminus (A_1 \cup \cdots \cup A_{\ell} \cup M),\]
	and so $\eta(a) = 0$. But this implies that
	\begin{equation}\label{aux-lem1-eq11}
		\mu (a) = \tau(a) = \sum_{j = \ell + 1}^{m} \chi_{A_j}(a).
	\end{equation}
	Since $a \notin (A_1 \cup \cdots \cup A_{\ell
	})$, it follows from the definitions of $B_j$ and $A_j^1$ that
	\[\mu (a) = \tau(a) = \sum_{j = 1}^{\ell} \chi_{B_j}(a) \leq \sum_{j = 1}^{\ell} \chi_{A_j^1}(a),\]
	which is again a contradiction. Therefore, if $\sum_{j = 1}^{\ell} \chi_{A_j^1}(a) < \mu(a)$ for some $a \in A$, then 
	\[a \in \biggl((A_1 \cup \cdots \cup A_{\ell}) \cap (A_{\ell + 1} \cup \cdots \cup A_m)\biggr)\setminus M.\]
	This proves \eqref{aux-lem1-eq4}.
\end{proof}

\begin{lemma}\label{aux-lem2}
	Let $\ell$ and $m$ be integers with $2 \leq \ell \leq m$. Let $\mathcal{A} = (A_1, \dots, A_m)$ be a finite sequence of finite subsets of a group $G$. Let $a \in A = A_1 \cup \cdots \cup A_m$ be an element satisfying the inequality
	\[1 \leq \eta(a) \leq \mu(a) < \ell.\]
	Then there exists the smallest nonnegative integer $\mathcal{L}(a)$ such that
	\begin{equation}\label{aux-lem2-eq1}
		\mathcal{L}(a) + \sum_{j = \mathcal{L}(a) + 1}^{\ell} \chi_{A_j}(a) = \mu(a).
	\end{equation}
Let $X$ and $\lambda(a)$ be defined as in Lemma \ref{aux-lem1}. Then for each $a \in X$, we have
\begin{equation}\label{aux-lem2-eq2}
	\mathcal{L}(a) > \tau(a).
\end{equation}
\end{lemma}

\begin{proof}
	For $a \in A$, let $P : [0, \mu (a)] \rightarrow [1, \ell]$ be the function defined by
	\begin{equation}\label{aux-lem2-eq3}
		P(i) = i + \sum_{j = i + 1}^{\ell} \chi_{A_j}(a).
	\end{equation}
Clearly,
	\[P(0) = \eta (a) ~\text{and}~ \mu (a) \leq P(\mu (a)).\]
For each $i \in [0, \mu (a) - 1]$, we have
	\begin{align*}
		P(i) = i + \sum_{j = i + 1}^{\ell} \chi_{A_j}(a)
		& \leq i + \biggl( 1 + \sum_{j = i + 2}^{\ell} \chi_{A_j}(a) \biggr)\\
		&= P(i + 1) \leq i + 1 + \sum_{j = i + 1}^{\ell} \chi_{A_j}(a) = P(i) + 1.
	\end{align*}
Hence $P(i) \leq P(i + 1) \leq P(i) + 1$ for each $i \in [0, \mu (a) - 1]$, and so $P$ is an increasing function. Clearly, if $P(i) < P(j)$ for some $i, j \in [1, \mu (a)]$, then $i < j$. Furthermore, $P(j) = P(j + 1)$ for some $j \in [0, \mu (a) - 1]$ if and only if $a \in A_{j + 1}$.
	
Now we show that $P([0, \mu (a)]) = [\eta(a), P(\mu (a))]$. Suppose there exist smallest integer $t$ such that $t \in [\eta(a), P(\mu (a))]$ and $P(i) \neq t$ for each $i \in [0, \mu(a)]$. Clearly, 
	\[\eta (a) < t < P(\mu (a)).\] 
	Let $i_0$ be the largest integer such that $P(i_0) = t - 1$. Then $P(i_0) <  P(i_0 + 1)$. Since $P(i) \leq P(i + 1) \leq P(i) + 1$ for each $i \in [0, \mu(a) - 1]$, it follows that 
	\[P(i_0) = t - 1 < t \leq P(i_0 + 1) \leq P(i_0) + 1 = t,\]
This implies that $P(i_0 + 1) = t$, which is a contradiction. Therefore,
	\[P([0, \mu (a)]) = [\eta(a), P(\mu (a))].\]
	Since $\mu (a) \in [\eta(a), P(\mu (a))]$, it follows that there exists $r \in [0, \mu (a)]$ such that 
	\[P(r) = r + \sum_{j = r + 1}^{\ell} \chi_{A_j}(a) = \mu (a).\]
	Let 
	\[S = \{ i \in [0, \mu (a)] : P(i) = \mu (a)\}.\]
	Then $r \in S$. Since $S$ is nonempty finite set, it follows that $\min(S)$ exists, and so 
	\[\mathcal{L}(a) = \min (S).\]
This proves \eqref{aux-lem2-eq1}.

Now we prove the second part of the lemma. If $a \in X$, then it follows from Lemma \ref{aux-lem1} (see \eqref{aux-lem1-eq5}) that
\[ \tau (a) + \sum_{j = \tau(a) + 1}^{\ell} \chi_{A_j}(a) < \mu(a).\]
Hence it follows from \eqref{aux-lem2-eq3} that
	\[P(\tau(a)) = \tau (a) + \sum_{j = \tau(a) + 1}^{\ell} \chi_{A_j}(a) < \mu(a) \leq P(\mu (a)),\]
	and so
	\[P (\tau (a)) < \mu (a) = P (\mathcal{L}(a)).\]
	This implies that
    \[\mathcal{L}(a) > \tau (a).\]
	This proves \eqref{aux-lem2-eq2}. 
\end{proof}

\begin{lemma}\label{aux-lem-existance}
	Let $\ell$ and $m$ be positive integers such that $2 \leq \ell < m$. Let $\mathcal{A} = (A_1, \ldots, A_m)$ be a finite sequence of nonempty finite subsets of an arbitrary group $G$ written multiplicatively. Let $A = A_1 \cup \cdots \cup A_m$, and $M = \{a \in A: \mu(a) = \ell\}$. For each $j \in [1, \ell]$, let 
	\[ B_j = \{a \in (A_{\ell + 1} \cup \cdots \cup A_m) \setminus M: \tau(a) \geq j\},\]
and
\[A_j^1 = A_j \cup B_j \cup M.\]
Then for each $j \in [1, \ell]$, there exist sets $A_j^2$ and $S_j$ satisfying the following conditions:
\begin{enumerate}
\item
\begin{equation}\label{aux-lem-existance-eq1}
  A_j^2 = A_j^1 \cup S_j.
\end{equation}
\item
\begin{equation}\label{aux-lem-existance-eq2}
 S_j =  A_j^2 \setminus A_j^1.
\end{equation}
\item 
\begin{equation}\label{aux-lem-existance-eq3}
   	   \mu (a) = \sum_{j = 1}^{\ell} \chi_{A_j^2}(a)~\text{for each}~ a \in A.
   \end{equation}   
\item
\begin{equation}\label{aux-lem-existance-eq4}
	|A_1^2| + \cdots + |A_{\ell}^2| = \sum_{a \in A} \mu(a).
\end{equation}
\item
\begin{equation}\label{aux-lem-existance-eq5}
 A_1^2 \cdots A_{\ell}^2 \subseteq \Pi^{\ell}(A).
\end{equation}
\item Moreover, for each $j \in [1, \ell]$,
\begin{equation}\label{aux-lem-existance-eq6}
   |A_j^2| \geq |A_j|.
\end{equation}
\end{enumerate}
\end{lemma}

\begin{proof} If $\ell = m$, then we define $A_j^2 = A_j$ for each $j \in [1, \ell]$. In this case, the proof is obvious. Now we assume that $\ell < m$. If $a \in A$, then it follows from \eqref{aux-lem1-eq3} that
	\[\mu(a) \geq \sum_{j = 1}^{\ell} \chi_{A_j^1}(a).\]
To construct the desired sets $A_j^2$ for $j \in [1, \ell]$, we consider the following cases:
	
	\noindent {\textbf{Case 1}} ($\mu(a) = \sum_{j = 1}^{\ell} \chi_{A_j^1}(a)$ for each $a \in A$). In this case, we define $S_j = \emptyset$ and $A_j^2 = A_J^1 \cup S_j = A_j^1$ for $j \in [1, \ell]$. Then each of the sets $A_1^2, \ldots, A_{\ell}^2$ clearly satisfy \eqref{aux-lem-existance-eq1}, \eqref{aux-lem-existance-eq2} and \eqref{aux-lem-existance-eq6}. Furthermore, it is easy to see that $A = A_1^1 \cup \cdots \cup A_{\ell}^1$.
\begin{align*}
\sum_{a \in A}\mu(a) =  \sum_{a \in A}\sum_{j = 1}^{\ell} \chi_{A_j^1}(a) = \sum_{a \in A_1^1 \cup \cdots \cup A_{\ell}^1}\sum_{j = 1}^{\ell} \chi_{A_j^1}(a) &= \sum_{j = 1}^{\ell} \sum_{a \in A_1^1 \cup \cdots \cup A_{\ell}^1} \chi_{A_j^1}(a)\\
 &= |A_1^1| + \cdots + |A_{\ell}^1|\\
 & = |A_1^2| + \cdots + |A_{\ell}^2|
\end{align*}
which proves \eqref{aux-lem-existance-eq4}	
	
\noindent {\textbf{Case 2}} ($\mu(a) > \sum_{j = 1}^{\ell} \chi_{A_j^1}(a)$ for some $a \in A$). In this case, it follows from Lemma \ref{aux-lem-existance} that
$a \in \biggl((A_1 \cup \cdots \cup A_{\ell}) \cap (A_{\ell + 1} \cup \cdots 
	\cup A_m)\biggr)$ and $a \not\in M$. 
Let
	\[X = \biggl\{a \in \biggl((A_1 \cup \cdots \cup A_{\ell}) \cap (A_{\ell + 1} \cup \cdots 
	\cup A_m)\biggr) \setminus M : \sum_{j = 1}^{\ell} \chi_{A_j^1}(a) < \mu(a) \biggr\}.\]
Let $\mathcal{L}(a)$ be defined as in Lemma \ref{aux-lem2} (see \eqref{aux-lem2-eq1}). Then it follows from \eqref{aux-lem2-eq2} that $\mathcal{L}(a) > \tau(a)$.
From the definition of the sets $B_j$, we know that an element $a \in X$ lies precisely in each of the sets $B_1, \ldots, B_{\tau(a)}$, and so it lies in each of the sets $A_1^1, \ldots, A_{\ell}^1$. Now we apply the following process to construct the sets $A_1^2, \ldots, A_{\ell}^2$ from the sets $A_1^1, \ldots, A_{\ell}^1$:

\textbf{Step-1:} Pick an element $a \in X$.\\

\textbf{Step-2:} Put the element $a$ in each of the sets $A_{\tau(a) + 1}^1, \ldots, A_{\mathcal{L}(a)}^1$.\\

\textbf{Step-3:} Repeat this process for each $a \in X$.\\

\textbf{Step-4:} Once Step-1 to Step-3 are completed for each $a \in X$, the sets $A_1^1, \ldots, A_{\ell}^1$ are transformed into new sets. Name these new sets obtained from $A_1^1, \ldots, A_{\ell}^1$ as $A_1^2, \ldots, A_{\ell}^2$, respectively.\\

Define $S_j = A_j^2 \setminus A_j^1$ for $j \in [1, \ell]$. Note that for each $j \in [1, \ell]$, we have $S_j \subseteq X$ and
	\begin{equation*}
	    A_j^2 = A_j^1 \cup S_j = A_j \cup B_j \cup S_j \cup M.
	\end{equation*}
Then each of the sets $A_1^2, \ldots, A_{\ell}^2$ clearly satisfy \eqref{aux-lem-existance-eq1}, \eqref{aux-lem-existance-eq2} and \eqref{aux-lem-existance-eq6}. 

Thus in both the cases, we have constructed the sets $A_1^2, \ldots, A_{\ell}^2$ which satisfy \eqref{aux-lem-existance-eq1}, \eqref{aux-lem-existance-eq2} and \eqref{aux-lem-existance-eq6}. It remains to show that each of these sets satisfy \eqref{aux-lem-existance-eq4} and \eqref{aux-lem-existance-eq5}.

First we prove \eqref{aux-lem-existance-eq4}. To prove this, first we show that $\mu(a) = \sum_{j = 1}^{\ell} \chi_{A_j^2}(a)$ for each $a \in A$. We consider two cases.
 
\noindent\textbf{Case-1}($a \not\in X$). It follows from \eqref{aux-lem1-eq5} that for each $a \in X$,
	\[\mu(a) >  \tau(a) + \sum_{j = \tau(g) + 1}^{\ell} \chi_{A_j}(a).\]
	If $a \in A \setminus X$, then
	\begin{equation}\label{aux-lem-existance-eq7}
		\mu(a) = \sum_{j = 1}^{\ell} \chi_{A_j^1}(a).
	\end{equation} 
Since for each $j \in [1, \ell]$, we know that $S_j \subseteq X$ and $S_j \cap A_j^1 = \emptyset$, it follows that if $a \not \in X$, then $a \notin S_j$ and so it follows from \eqref{aux-lem-existance-eq1} that
    \[\chi_{A_j^1}(a) = \chi_{A_j^2}(a)\]
Therefore, it follows from \eqref{aux-lem-existance-eq7} that
     \begin{equation}\label{aux-lem-existance-eq8}
     	 \mu (a) = \sum_{j = 1}^{\ell} \chi_{A_j^2}(a)~ \text{for each}~ a \not\in X.
     \end{equation}
     
\noindent\textbf{Case-2 ($a \in X$).} In this case, it follows from the construction of the sets $A_1^2, \ldots, A_{\ell}^2$ that if $a \in X$, then $a \in A_j^2$ for each $j \in [\tau(a) + 1, \mathcal{L}(a)]$. Furthermore, we already know that $a \in B_j \subseteq A_j^2$ for each $j \in [1, \tau(a)]$. Hence
\[\sum_{j = 1}^{\mathcal{L}(a)} \chi_{A_j^2}(a) = \sum_{j = 1}^{\tau(a)} \chi_{A_j^2}(a) + \sum_{j = \tau(a) + 1}^{\mathcal{L}(a)} \chi_{A_j^2}(a) = \tau(a) + (\mathcal{L}(a) - \tau(a)) = \mathcal{L}(a).\]
	Therefore,
	\begin{equation}\label{aux-lem-existance-eq9}
		\mathcal{L}(a) = \sum_{j = 1}^{\mathcal{L}(a)} \chi_{A_j^2}(a)~ \text{for each }~ a \in X.
	\end{equation}
While constructing the sets $A_1^2, \ldots, A_{\ell}^2$, an element $a \in X$ was put only in the sets $A_{\tau(a) + 1}^1, \ldots, A_{\mathcal{L}(a)}^1$, and so $a \notin S_j$ for each $j \in [\mathcal{L}(a) + 1, \ell]$ which implies $\chi_{A_j^2}(a) = \chi_{A_j^1}(a)$ for each $j \in  [\mathcal{L}(a) + 1, \ell]$. Hence
\begin{equation}\label{aux-lem-existance-eq10}
		\sum_{j = \mathcal{L}(a) + 1}^{\ell} \chi_{A_j^2}(a) = \sum_{j = \mathcal{L}(a) + 1}^{\ell} \chi_{A_j^1}(a).
	\end{equation}
Furthermore, $a \not\in B_j$ for $j \geq \tau(a) + 1$, and if $a \in X$, then $a \not\in M$ also. This implies that $\chi_{A_j^1}(a) = \chi_{A_j}(a)$ for each $j \in  [\mathcal{L}(a) + 1, \ell]$. Hence it follows from \eqref{aux-lem-existance-eq10} that
	\begin{equation}\label{aux-lem-existance-eq11}
		\sum_{j = \mathcal{L}(a) + 1}^{\ell} \chi_{A_j^2}(a) = \sum_{j = \mathcal{L}(a) + 1}^{\ell} \chi_{A_j}(a)~ \text{for each } a\in X.
	\end{equation}

It follows from the definition of $\mathcal{L}(a)$ (see \eqref{aux-lem2-eq1}) that for each $a \in X$,
\[\mu(a) = \mathcal{L}(a) + \sum_{j = \mathcal{L}(a) + 1}^{\ell} \chi_{A_j}(a).\]
Substituting the value of $\mathcal{L}(a)$ from \eqref{aux-lem-existance-eq9}, we get
\[\mu(a) = \sum_{j = 1}^{\mathcal{L}(a)} \chi_{A_j^2}(a) + \sum_{j = \mathcal{L}(a) + 1}^{\ell} \chi_{A_j}(a).\]
Since $A_j \subseteq A_j^1 \subseteq A_j^2$ for each $j \in [1, \ell]$, it follows that $\chi_{A_j}(a) = \chi_{A_j^2}(a)$, and so the above equation implies that
\[\mu(a) = \sum_{j = 1}^{\mathcal{L}(a)} \chi_{A_j^2}(a) + \sum_{j = \mathcal{L}(a) + 1}^{\ell} \chi_{A_j^2}(a).\]
Now it follows from the above equation and \eqref{aux-lem-existance-eq11} that
\begin{equation}\label{aux-lem-existance-eq12}
  \mu(a) = \sum_{j = 1}^{\ell} \chi_{A_j^2}(a) ~\text{for each}~ a \in X.
\end{equation}
Hence it follows from \eqref{aux-lem-existance-eq7} and \eqref{aux-lem-existance-eq12} that
   \begin{equation}\label{aux-lem-existance-eq13}
   	   \mu (a) = \sum_{j = 1}^{\ell} \chi_{A_j^2}(a)~\text{for each}~ a \in A.
   \end{equation} 
This proves \eqref{aux-lem-existance-eq3}.  
Therefore,
	\[\sum_{a \in A} \mu(a) = \sum_{a \in A} \sum_{j = 1}^{\ell} \chi_{A_j^2}(a) = \sum_{j = 1}^{\ell} \sum_{a \in A}  \chi_{A_j^2}(a)= |A_1^2| + \cdots + |A_{\ell}^2|,\]
which proves \eqref{aux-lem-existance-eq4}.

To prove \eqref{aux-lem-existance-eq5}, first we prove the following claim.
	
	\noindent {\textbf{Claim 1}} (If $a \in S_r$ for some $r \in [2, \ell]$, then there exist at least $r$ sets among $A_1, \ldots, A_{r -  1}, A_{\ell + 1}, \ldots, A_m$ that contains $a$.) In this case, $a \in A_r^2 \setminus A_r^1$, and so $a \notin M$ which implies $\mu(a) < \ell$. To prove the claim, it suffices to prove that 
\begin{equation}\label{aux-lem-existance-eq14}
  \sum_{j = 1}^{r - 1} \chi_{A_j}(a) + \sum_{j = \ell + 1}^{m} \chi_{A_j}(a) \geq r.
\end{equation}
Suppose that 
	\[\sum_{j = 1}^{r - 1} \chi_{A_j}(a) + \sum_{j = \ell + 1}^{m} \chi_{A_j}(a) < r.\]
This implies that
\begin{equation}\label{aux-lem-existance-eq15}
	\mu (a) = \sum_{j = 1}^{m} \chi_{A_j}(a) = \left(\sum_{j = 1}^{r - 1} \chi_{A_j}(a) + \sum_{j = \ell + 1}^{m} \chi_{A_j}(a)\right) + \sum_{j = r}^{\ell} \chi_{A_j}(a)  < r + \sum_{j = r}^{\ell} \chi_{A_j}(a).
\end{equation}
	Since $a \not \in A_r^1$, it follows that $a \not\in A_r$. Note that $a \in B_j \subseteq A_j^2$ for all $j \in [1, \tau(a)]$. It follows from the construction of the sets $A_j^2$ that $a \in A_j^2$ for all $j \in [1, \mathcal{L}(a)]$. Therefore, it follows from \eqref{aux-lem-existance-eq15} that
	\begin{align*}
		\mu (a) < r + \sum_{j = r}^{\ell} \chi_{A_j}(a) &= r + \sum_{j = r + 1}^{\ell} \chi_{A_j}(a)\\
		&= r + \sum_{j = r + 1}^{\mathcal{L}(a)} \chi_{A_j}(a) + \sum_{j = \mathcal{L}(a) + 1}^{\ell} \chi_{A_j}(a)\\
		&\leq r + \sum_{j = r + 1}^{\mathcal{L}(a)} \chi_{A_j^2}(a) + \sum_{j = \mathcal{L}(a) + 1}^{\ell} \chi_{A_j}(a)\\
		&= r + (\mathcal{L}(a) - r) + \sum_{j = \mathcal{L}(a) + 1}^{\ell} \chi_{A_j}(a)\\
		&= \mathcal{L}(a) + \sum_{j = \mathcal{L}(a) + 1}^{\ell} \chi_{A_j}(a)\\
		&= \mu (a) 
	\end{align*}
	which is impossible. Hence
	\[\sum_{j = 1}^{r - 1} \chi_{A_j}(a) + \sum_{j = \ell + 1}^{m} \chi_{A_j}(a) \geq r\]
	which proves \eqref{aux-lem-existance-eq14} and hence Claim $1$.
	
Now we prove \eqref{aux-lem-existance-eq5}. Let $x = x_1 \cdots x_{\ell} \in A_1^2 \cdots A_{\ell}^2$ such that $x_j \in A_j^2$ for each $j \in [1, \ell]$. We show that for each $j \in [1, \ell]$, there exists $t_j \in ([1, j] \cup [\ell + 1, m]) \setminus \{t_1, \ldots, t_{j - 1}\}$ such that $x_j \in A_{t_j}$. For $j = 1$, since $x_1 \in A_1^2 = A_1^1 = A_1 \cup A_{\ell + 1} \cdots \cup A_m$, it follows that there exits $t_1 \in \{1\} \cup [\ell + 1, m]$ such that $x_1 \in A_{t_1}$. Now assume that $j \geq 2$. Suppose we have already chosen distinct integers $t_1, \ldots, t_{j - 1}$. We choose $t_j$ as follows:
	Since $x_j \in A_j^2 = A_j \cup B_j \cup M \cup S_j$, it follows that 
	\[~\text{either}~ x_j \in A_j ~\text{or}~ x_j \in B_j ~\text{or}~ x_j \in M ~\text{or}~ x_j \in S_j.\]
		We choose $t_j$ in each case as follows:
	\begin{enumerate}
		\item If $x_j \in A_j$, then we choose $t_j = j$.
		\item It $x_j \in B_j$, then $\tau(x_j) \geq j$. Since $\tau(x_j) \geq j$, it follows from the definition of $\tau(x_j)$ that $x_j$ lies in at least $j$ sets among the sets $A_{\ell + 1}, \ldots, A_m$. Therefore, we can choose $t_j \in [\ell + 1, m] \setminus \{t_1, \ldots, t_{j - 1}\}$ such that $x_j \in A_{t_j}$.
		\item It $x_j \in M$, then $\mu(x_j) = \ell$. Since $\mu(x_j) = \ell$, it follows from the definition of $\mu(x_j)$ that $x_j$ lies in at least $\ell$ sets among the sets $A_1, \ldots, A_m$, and so there exist $t_j \in ([1, j] \cup [\ell + 1, m]) \setminus \{t_1, \ldots, t_{j - 1}\}$ such that $x_j \in A_{t_j}$.
		\item If $x_j \in S_j$, then it follows from Claim $1$ that there exists $t_j \in ([1, j - 1] \cup [\ell + 1, m]) \setminus \{t_1, \ldots, t_{j - 1}\}$ such that $x_j \in A_{t_j}$.
	\end{enumerate}
Thus there exist $\ell$ distinct sets $A_{t_1}, \ldots, A_{t_{\ell}}$ such that $x_j \in A_{t_j}$ for each $j \in [1, \ell]$. Hence
	\[x = x_1 \cdots x_{\ell} \in A_{t_1} \cdots A_{t_{\ell}} \subseteq \Pi^{\ell}(\mathcal{A}).\]
Therefore,
	\[A_1^2 \cdots A_{\ell}^2  \subseteq \Pi^{\ell}(\mathcal{A}),\]
which proves \eqref{aux-lem-existance-eq5}. 
\end{proof}

\begin{lemma}\label{aux-lem-inv-tf}
	Let $\ell$ and $m$ be positive integers such that $2 \leq \ell \leq m$. Let $\mathcal{A} = (A_1, \ldots, A_m)$ be a finite sequence of finite subsets of torsion-free group $G$, and let $A = A_1 \cup \cdots \cup A_m$. Let $|A_i|  \geq 2$ for each $i  \in [1, m]$ and
	\[|\Pi^{\ell}(\mathcal{A})| =  \sum_{a \in A} \mu (a) - \ell + 1.\] 
Let $M = \{a \in A : \mu(a) = \ell\}$. Then the sets $A_1 \cup M, \ldots, A_m \cup M$ are geometric progressions of types $(a_1, g, b_1), \ldots, (a_m, g, b_m)$, respectively for some $a_i, b_i, g \in G$ for $i = 1, \ldots, m$ and $g \neq 1$.
\end{lemma}

\begin{proof}	
	For $\ell = 2$, the lemma follows from Lemma \ref{aux-lem-special-case}. Now we assume that $\ell \geq 3$. Let $A_1^2, \ldots,  A_{\ell}^2$ be the sets as defined in Lemma \ref{aux-lem-existance}. Then it follows from Lemma \ref{aux-lem-existance} that
	\[A_1^2 \cdots A_{\ell}^2 \subseteq \Pi^{\ell}(\mathcal{A}),\]
	and
	\[|A_1^2| + \cdots + |A_{\ell}^2| =  \sum_{a \in A} \mu (a).\]
For convenience, for each $j \in [1, m]$, we denote the set $A_j \cup M$ by $A_j^0$. It is easy to see that $A_{\ell}^2 = A_{\ell}^1 = A_{\ell}^0 = A_{\ell} \cup M$.
	
\noindent {\textbf{Claim 1}} ($A_1^2, \ldots, A_{\ell - 1}^2, A_{\ell}^2$ are geometric progressions with the same common ratio). Since $|\Pi^{\ell}(\mathcal{A})| =  \sum_{a \in A} \mu (a) - \ell + 1$, it follows from Lemma \ref{aux-lem-existance} and Lemma \ref{kemp-gen-thm-tf} that
	\begin{align*}
	     \sum_{a \in A} \mu (a) - \ell + 1 &= 	|\Pi^{\ell}(\mathcal{A})|\\
	    &\geq |A_1^2 \cdots A_{\ell}^2|\\
	    &\geq |A_1^2| + \cdots + |A_{\ell}^2| - \ell + 1\\
	    &=  \sum_{a \in A} \mu (a) - \ell + 1.
	\end{align*}
	Hence 
	\[|A_1^2 \cdots A_{\ell}^2| = |A_1^2| + \cdots + |A_{\ell}^2| - \ell + 1.\]
	Since $|A_j| \geq 2$ for each $j \in [1, m]$, it follows from \eqref{aux-lem-existance-eq1} that $|A_j^2| \geq 2$ for each $j \in [1, \ell]$. Therefore, it follows from Lemma \ref{brail-gen-inv-thm-tf} that the sets $A_1^2, \ldots, A_{\ell - 2}^2, A_{\ell - 1}^2, A_{\ell}^2 = A_{\ell}^1 = A_{\ell}^0 = A_{\ell} \cup M$ are geometric progressions of type $(\alpha_1, g, \beta_1),\ldots, (\alpha_{\ell}, g, \beta_{\ell})$, respectively for some $\alpha_i, \beta_i, g \in G$, where $\alpha_{i + 1} = \beta_i^{- 1}$ for each $i \in [1, \ell - 1]$ and $g \neq 1$. This completes the proof of Claim $1$.

\noindent {\textbf{Claim 2}} (If $a \in A$ be an element such that $a \in A_{\ell} \setminus A_{\ell - 1}^0$, then $a \not\in A_{\ell - 1}^2$).
First note that if $a \in A$ such that $a \not \in A_{\ell}^0$, then $a \not\in A_{\ell}^2$ as $A_{\ell}^2 = A_{\ell}^0$. Now suppose that $a \in A_{\ell - 1}^2$. Since $A_{\ell - 1}^2 = A_{\ell - 1}^0 \cup B_{\ell - 1} \cup S_{\ell - 1}$ and $a \not \in A_{\ell - 1}^0$, it follows that $\mu(a) < \ell$ and
		\[~\text{either}~ a \in B_{\ell - 1} ~\text{or}~ a \in S_{\ell - 1}.\]
If $a \in B_{\ell - 1}$, then $\tau (a) \geq \ell - 1$, and so it follows from the definition of $\tau(a)$ that
		\begin{align*}
			\mu(a) = \sum_{j = 1}^m \chi_{A_j}(a) &= \sum_{j = 1}^{\ell} \chi_{A_j}(g) + \sum_{j = \ell + 1}^m \chi_{A_j}(a)\\
			&\geq \chi_{A_{\ell}}(a) + \sum_{j = \ell + 1}^m \chi_{A_j}(a)\\
			&\geq \chi_{A_{\ell}}(a) + \tau(a)\\
			&\geq 1 + (\ell - 1)\\
			&= \ell,
		\end{align*}
		which is a contradiction. Hence $a \in S_{\ell - 1}$. In this case, $\mathcal{L} (a) \geq \ell - 1$. Since $a \in A_{\ell}$, it follows from  \eqref{aux-lem2-eq1} that 
	\[\mu(a) = \mathcal{L}(a) + \sum_{j = \mathcal{L}(a) + 1}^{\ell} \chi_{A_j}(a) \geq \mathcal{L}(a) + \chi_{A_{\ell}}(a) \geq (\ell - 1) + 1 = \ell,\]
		which is again a contradiction. Therefore,  $a \not \in A_{\ell - 1}^2$ which proves Claim $2$. 

\noindent {\textbf{Claim 3}} (If $C = A_{\ell - 1}^2 \setminus A_{\ell - 1}^0$, then $A_{\ell}^0 \cap C = \emptyset$).
Clearly, $A_{\ell - 1}^0 \cap C = \emptyset$, and since $A_{\ell - 1}^0 = A_{\ell - 1} \cup M$, it follows that $M \cap C = \emptyset$. Therefore, it is enough to show that $A_{\ell} \cap C = \emptyset$. Let $a \in A_{\ell} \cap C$. In this case, since $a \in A_{\ell}$ and $a \not \in A_{\ell - 1}^0$, it follows from Claim that $a \notin A_{\ell - 1}^2$. Hence $a \not \in C$, which is a contradiction. Therefore, $A_{\ell} \cap C = \emptyset$ and this proves Claim $3$.
	
Now we consider the sets $A_1^2, \ldots, A_{\ell - 2}^2$, $A_{\ell}^0 \cup C$ and $A_{\ell - 1}^0$. 
	
\noindent {\textbf{Claim 4}} ($A_1^2 \cdots A_{\ell - 2}^2 (A_{\ell}^0 \cup C) A_{\ell - 1}^0 \subseteq \Pi^{\ell}(\mathcal{A})$). Let 
	\[x = x_1 \cdots x_{\ell} \in A_1^2 \cdots A_{\ell - 2}^2 (A_{\ell}^0 \cup C) A_{\ell - 1}^0,\] 
	where $x_j \in A_j^2$ for each $j \in [1, \ell - 2]$, $x_{\ell - 1} \in A_{\ell}^0 \cup C$ and $x_{\ell} \in A_{\ell - 1}^0$. Similar to the proof of \eqref{aux-lem-existance-eq5} of Lemma \ref{aux-lem-existance}, we can show that there exist distinct sets $A_{t_1}, \ldots, A_{t_{\ell - 2}}$ such that $x_j \in A_{t_j}$, where $t_j \in [1, m] \setminus \{\ell - 1, \ell\}$. Since $x_{\ell - 1} \in (A_{\ell}^0 \cup C)$, it follows that 
	\[~\text{either}~ x_{\ell - 1} \in A_{\ell}^0 ~\text{or}~ x_{\ell - 1} \in C.\] 
	If $x_{\ell - 1} \in A_{\ell}^0$, then either $x_{\ell - 1} \in A_{\ell}$ or $x_{\ell - 1} \in M$. This implies that either $x_{\ell - 1} \in A_{\ell}$ or $\mu(x_{\ell - 1}) = \ell$. In both cases, there exits $t_{\ell - 1} \in [1, m] \setminus \{t_1, \ldots, t_{\ell - 2}, \ell - 1\}$ such that $x_{\ell - 1} \in A_{t_{\ell - 1}}$. Now assume that $x_{\ell - 1} \notin A_{\ell}^0$. Then $x_{\ell - 1} \in C \subseteq B_{\ell - 1} \cup S_{\ell - 1}$. We observe the following:
	\begin{enumerate}
		\item If $x_{\ell - 1} \in B_{\ell - 1}$, then $\tau (x_{\ell - 1}) = \ell - 1$ as $x_{\ell - 1} \not\in M$. Hence there exists $t_{\ell - 1} \in [\ell + 1, m] \setminus \{t_1, \ldots, t_{\ell - 2}, \ell - 1, \ell\}$ such that $x_{\ell - 1} \in A_{t_{\ell - 1}}$.
		\item If $x_{\ell - 1} \in S_{\ell - 1}$, then it follows from Lemma \ref{aux-lem-existance} (see Claim $2$) that there exists $t_{\ell - 1} \in ([1, \ell - 2] \cup [\ell + 1, m]) \setminus \{t_1, \ldots, t_{\ell - 2}\}$ as $x_{\ell - 1} \not\in A_{\ell - 1}^0$.
	\end{enumerate}
Thus in each situation, there exits $t_{\ell - 1} \in [1, m] \setminus \{t_1, \ldots, t_{\ell - 2}, \ell - 1\}$ such that $x_{\ell - 1} \in A_{t_{\ell - 1}}$. 
	
Next, since $x_{\ell} \in A_{\ell - 1}^0 = A_{\ell - 1} \cup M$, it follows that either $x_{\ell} \in A_{\ell - 1}$ or $x_{\ell} \in M$. This implies that either $x_{\ell} \in A_{\ell - 1}$ or $\mu(x_{\ell}) = \ell$, and so there exits $t_{\ell} \in [1, m] \setminus \{t_1, \ldots, t_{\ell - 2}, t_{\ell - 1}\}$ such that $x_{\ell} \in A_{t_{\ell}}$. 

Thus we have shown that there exist distinct $\ell$ sets $A_{t_1}, \ldots, A_{t_{\ell}}$ among $A_1, \ldots, A_m$ such that $x_j \in A_{t_j}$ for each $j \in [1, \ell]$. Hence
	\[x = x_1 \cdots x_\ell \in A_{t_1} \cdots A_{t_\ell} \subseteq \Pi^{\ell}(\mathcal{A}).\]
	Therefore,
	\[A_1^2 \cdots A_{\ell - 2}^2 (A_{\ell}^0 \cup C) A_{\ell - 1}^0 \subseteq \Pi^{\ell}(\mathcal{A}),\]
which proves Claim $4$.

\noindent {\textbf{Claim 5}} ($A_1^2, \ldots, A_{\ell - 2}^2, A_{\ell}^0 \cup C, A_{\ell - 1}^0$ are geometric progressions of type $(\delta_1, g_1, \gamma_1), \ldots, (\delta_{\ell}, g_1, \gamma_{\ell})$, respectively for some $\delta_1, \gamma_1, \ldots, \gamma_{\ell}, g_1 \in G$ with $g_1 \neq 1$).
Since $|C| = |A_{\ell - 1}^2| - |A_{\ell - 1}^0|$ and $A_{\ell}^0 = A_{\ell}^2$, it follows from Claim $3$ that
\begin{align*}
 |A_1^2| + \cdots + |A_{\ell - 2}^2| + |(A_{\ell}^0 \cup C)| + |A_{\ell - 1}^0| &= |A_1^2| + \cdots + |A_{\ell - 2}^2| + |A_{\ell}^0| + (|C| + |A_{\ell - 1}^0|)\\ 
 & =  |A_1^2| + \cdots + |A_{\ell - 2}^2| + |A_{\ell - 1}^2| + |A_{\ell}^2| \\
 &= \sum_{a \in A} \mu (a).
\end{align*}
	\[\]
Now it follows from Claim $4$ and Lemma \ref{kemp-gen-thm-tf} that
	\begin{align*}
		\sum_{a \in A} \mu (a) - \ell + 1 &= 	|\Pi^{\ell}(\mathcal{A})|\\
		&\geq |A_1^2 \cdots A_{\ell - 2}^2 (A_{\ell}^0 \cup C) A_{\ell - 1}^0|\\
		&\geq |A_1^2| + \cdots + |A_{\ell - 2}^2| + |A_{\ell}^0 \cup C| + |A_{\ell - 1}^0| - \ell + 1\\
		&=  \sum_{a \in A} \mu (a) - \ell + 1. 
	\end{align*}
	Hence 
	\[|A_1^2 \cdots A_{\ell - 2}^2 (A_{\ell}^0 \cup C) A_{\ell - 1}^0| = |A_1^2| + \cdots + |A_{\ell - 2}^2| + |A_{\ell}^0 \cup C| + |A_{\ell - 1}^0| - \ell + 1.\]
Therefore, it follows from Lemma \ref{brail-gen-inv-thm-tf} that the sets $A_1^2, \ldots, A_{\ell - 2}^2, A_{\ell}^0 \cup C, A_{\ell - 1}^0$ are geometric progressions of type $(\delta_1, g_1, \gamma_1), \ldots, (\delta_{\ell}, g_1, \gamma_{\ell})$, respectively for some $\delta_1, \gamma_1, \ldots, \gamma_{\ell}, g_1 \in G$ with $g_1 \neq 1$.
	
\noindent {\textbf{Claim 6}} ( $A_{\ell - 1}^0$ and $A_{\ell}^0$ are geometric progressions of type $(\delta_{\ell}', g, \gamma_{\ell}')$ and $(\alpha_{\ell}, g, \beta_{\ell})$ for some $\delta_{\ell}', \gamma_{\ell}', \alpha_{\ell}, \beta_{\ell}, g \in G$ with $g \neq 1$). Since $\ell \geq 3$, it follows from Claim $1$ that the $A_1^2$ and $A_{\ell}^2 = A_{\ell}^0$ are geometric progressions of type $(\alpha_1, g, \beta_1)$ and $(\alpha_{\ell}, g, \beta_{\ell})$, respectively for some $\alpha_1, \alpha_{\ell}, \beta_1, \beta_{\ell}, g \in G$ with $g \neq 1$.  It follows from Claim $5$ that $A_1^2$ and $A_{\ell - 1}^0$ are geometric progressions of type $(\delta_1, g_1, \gamma_1)$ and $(\delta_{\ell}, g_1, \gamma_{\ell})$, respectively for some $\delta_1, \gamma_1, \delta_{\ell}, \gamma_{\ell}, g_1 \in G$ with $g_1 \neq 1$. Since $A_1^2$ is also a geometric progressions geometric progression of type $(\alpha_1, g, \beta_1)$, it follows from Lemma \ref{gp-lem2} that $A_{\ell - 1}^0$ can also be represented as a geometric progression of the type $(\delta_{\ell}', g, \gamma_{\ell}')$ for some $\delta_{\ell}', \gamma_{\ell}' \in G$. Thus $A_{\ell - 1}^0$ and $A_{\ell}^0$ are geometric progressions of type $(\delta_{\ell}', g_1, \gamma_{\ell}')$ and $(\alpha_{\ell}, g, \beta_{\ell})$ for some $\delta_{\ell}', \gamma_{\ell}', \alpha_{\ell}, \beta_{\ell}, g \in G$ with $g \neq 1$. This proves Claim $6$.

Now we show that the sets $A_1 \cup M, \ldots, A_m \cup M$ are geometric progressions of types $(a_1, g, b_1), \ldots, (a_m, g, b_m)$, respectively for some $a_i, b_i, g \in G$ for $i = 1, \ldots, m$ and $g \neq 1$. It follows from Claim $6$ that $A_{\ell - 1}^0$ and $A_{\ell}^0$ are geometric progressions of type $(\delta_{\ell}', g_1, \gamma_{\ell}')$ and $(\alpha_{\ell}, g, \beta_{\ell})$ for some $\delta_{\ell}', \gamma_{\ell}', \alpha_{\ell}, \beta_{\ell}, g \in G$ with $g \neq 1$. Since the sumset $\Pi^{\ell}(\mathcal{A})$ is invariant under the reordering of the sets $A_1, A_2, \ldots, A_m$, repeated application of Claim $6$ and Lemma \ref{gp-lem2} implies that the sets $A_1^0, \ldots, A_m^0$ are geometric progressions of types $(a_1, g, b_1), \ldots, (a_m, g, b_m)$, respectively for some $a_i, b_i, g \in G$ for $i = 1, \ldots, m$ and $g \neq 1$. Therefore, the sets $A_1 \cup M, \ldots, A_m \cup M$ are geometric progressions of types $(a_1, g, b_1), \ldots, (a_m, g, b_m)$, respectively for some $a_i, b_i, g \in G$ for $i = 1, \ldots, m$ and $g \neq 1$.This completes the proof.
\end{proof}

\begin{lemma}\label{aux-lem-special-case}
	Let $m$ be positive integers such that $m \geq 2$. Let $\mathcal{A} = (A_1, \ldots, A_m)$ be a finite sequence of nonempty finite subsets of a torsion-free group $G$ such that $|A_i| \geq 2$ for each $i \in [1, m]$. Let $A = A_1 \cup \cdots \cup A_m$. Then
\begin{equation}\label{aux-lem-special-case-eq}
|\Pi^{2}(\mathcal{A})| = \sum_{a \in A} \mu (a) - 1,
\end{equation}
    if and only if $\mathcal{A}$ is a $(2, g)$ minimizing sequence of sets for some $g \in G$. Furthermore, if the above equality holds, then $A_1 \cup M, \ldots, A_m \cup M$ are geometric progressions of type $(\alpha_1, g, \beta_1), \ldots, (\alpha_m, g, \beta_m)$, respectively for some $\alpha_i, \beta_i, g \in G$ for $i = 1, \ldots, m$ with $g \neq 1$, where $M = \{a \in A : \mu(a) = 2\}$.
\end{lemma}

\begin{proof} 
We prove the lemma separately for $m = 2$ and $m \geq 3$.

\noindent {\textbf{Case 1}}($m = 2$). In this case, $\mathcal{A} = (A_1, A_2)$, $A = A_1 \cup A_2$, $M \subseteq A_1$, $M \subseteq A_2$, and 
\[\Pi^{2}(\mathcal{A}) = A_1 A_2 \cup A_2 A_1.\]
If $\mathcal{A}$ is a $(2, g)$ minimizing sequence of sets for some $g \in G$, then it follows from Lemma \ref{aux-lem-inv3} that
 \[|\Pi^{2}(\mathcal{A})| = \sum_{a \in A} \mu (a) - 1.\]
 Conversely, assume that 
 \begin{equation}\label{aux-lem-special-case-eq1}
 |\Pi^{2}(\mathcal{A})| = \sum_{a \in A} \mu (a) - 1.
 \end{equation}
 Then since $A_1 A_2 \subseteq \Pi^{2}(\mathcal{A})$, and $\sum_{a \in A} \mu (a) = |A_1| + |A_2|$, it follows from \eqref{aux-lem-special-case-eq1} and Lemma \ref{kemp-gen-thm-tf} that
 \begin{align}\label{aux-lem-special-case-eq2}
 |A_1| + |A_2| - 1 = \sum_{a \in A} \mu (a) - 1 =  |\Pi^{2}(\mathcal{A})| \geq |A_1 A_2| \geq |A_1| + |A_2| - 1.
 \end{align} 
This implies that 
\[|A_1 A_2| = |A_1| + |A_2| - 1,\]
and so it follows from  Lemma \ref{brail-gen-inv-thm-tf} that $A_1$ and $A_2$ are geometric progressions of type $(\alpha_1, g, \beta_1)$ and $(\alpha_2, g, \beta_2)$, respectively for some $\alpha_1, \alpha_2, \beta_1, \beta_2, g \in G$ with $g \neq 1$. Let $B_1 = A_1$ and $B_2 = A_2$. Then it is easy to see that the sets $B_1$ and $B_2$ satisfy the conditions of Definition \ref{min-seq}, and hence the sequence $\mathcal{A} = (A_1, A_2)$ is a minimizing sequence of sets. Furthermore, since $A_1 \cup M = A_1$ and $A_2 \cup M = A_2$, it follows that $A_1 \cup M$ and $A_2 \cup M$ are geometric progressions of type $(\alpha_1, g, \beta_1)$ and $(\alpha_2, g, \beta_2)$, respectively for some $\alpha_1, \alpha_2, \beta_1, \beta_2, g \in G$ with $g \neq 1$. 

\noindent {\textbf{Case 2}}	($m \geq 3$). If $\mathcal{A}$ is a $(2, g)$ minimizing sequence of sets for some $g \in G$, then it follows from Lemma \ref{aux-lem-inv3} that
	\[|\Pi^{2}(\mathcal{A})| = \sum_{a \in A} \mu (a) - 1.\]
Conversely, assume that
 \begin{equation}\label{aux-lem-special-case-eq3}
 |\Pi^{2}(\mathcal{A})| = \sum_{a \in A} \mu (a) - 1.
 \end{equation}
We consider two subcases.

\noindent {\textbf{Subcase 2.1}} ($A_1 = \cdots = A_m$). In this case, we have
	\[\Pi^{2}(\mathcal{A}) = A_1 A_1,\]
	and so
	\[|A_1 A_1| = |\Pi^{2}(\mathcal{A})| = \sum_{a \in A} \mu (a) - 1 = |A_1| + |A_1| - 1.\]
Now the lemma follows from the argument similar to that in Case $1$.
	
	\noindent {\textbf{Subcase 2.2}} ($A_r \neq A_s$ for some $r, s \in [1, m]$). Let 
	\[A_r^0 = A_r \cup M, A_s^0 = A_s \cup M,\]
and let
\[C = \bigcup_{k \in [1, m] \setminus \{r, s\}} A_k.\]
Write $C^0 = C \cup M$.
	
	\noindent {\textbf{Claim 1}} ($|(A_r^0 \cup C^0)| + |A_s^0 | = \displaystyle\sum_{a \in A} \mu (a)$). Clearly,
\begin{align}\label{aux-lem-special-case-eq4}
		\sum_{a \in A} \mu (a) = \sum_{a \in A \setminus M} \mu (a) + \sum_{a \in M} \mu (a) &= \sum_{a \in A \setminus M} 1 + \sum_{a \in M} 2 \notag\\
 &= (|A \setminus M|) + 2|M| = |A| + |M|.
	\end{align}
	
	Now we show that $(A_r^0 \cup C^0) \cap A_s^0 = M$. Let $Y = \displaystyle\bigcup_{k \in [1, m] \setminus \{s\}} A_k$. Then 
\[A_r^0 \cup C^0 = Y \cup M,\]
	and so
	\[(A_r^0 \cup C^0) \cap A_s^0 = (Y \cup M) \cap (A_{s} \cup M) = (Y \cap A_{s}) \cup M.\]
It is obvious that if $a \in A_i \cap A_j$ for some $i, j \in [1, m]$ with $i \neq j$, then $a \in M$. Hence
	\[Y \cap A_{s} \subseteq M,\]
	and so
	\[(A_r^0 \cup C^0) \cap A_s^0 = (Y \cap A_{s}) \cup M = M.\]
	Therefore,
	\begin{align}\label{aux-lem-special-case-eq5}
		|(A_r^0 \cup C^0)| + |A_s^0|  = |(A_r^0 \cup C^0) \cup A_s^0 | + |(A_r^0 \cup C^0) \cap A_s^0 | &= |A_1 \cup \cdots \cup A_m \cup M| + |M| \notag \\
		&= |A| + |M|.
	\end{align}
	Therefore, it follows from \eqref{aux-lem-special-case-eq4} and \eqref{aux-lem-special-case-eq5} that
\begin{equation}\label{aux-lem-special-case-eq6}
 |(A_r^0 \cup C^0)| + |A_s^0 | = \sum_{a \in A} \mu (a)
\end{equation}
which proves Claim $1$. 
Now it follows from \eqref{aux-lem-special-case-eq3}, \eqref{aux-lem-special-case-eq6} and Lemma \ref{kemp-gen-thm-tf} that
\begin{align}\label{aux-lem-special-case-eq7}
\sum_{a \in A} \mu (a) - 1 =  |\Pi^{2}(\mathcal{A})| \geq |(A_r^0 \cup C^0) A_s^0| \geq |A_r^0 \cup C^0| + |A_s^0| - 1 = \sum_{a \in A} \mu (a) - 1.
 \end{align} 
Therefore, it follows from \eqref{aux-lem-special-case-eq7} and from the fact $(A_r^0 \cup C^0) A_s^0 \subseteq \Pi^{2}(\mathcal{A})$ that
\begin{equation}\label{aux-lem-special-case-eq8}
 |(A_r^0 \cup C^0) A_s^0| = |A_r^0 \cup C^0| + |A_s^0| - 1 
\end{equation}
and 
\begin{equation}\label{aux-lem-special-case-eq9}
(A_r^0 \cup C^0) A_s^0 = \Pi^{2}(\mathcal{A}).
\end{equation} 
Let $B_1 = A_r^0 \cup C^0$ and $B_2 = A_s^0$. Then $B_1 B_2 = \Pi^{2}(\mathcal{A})$. Also, it follows from \eqref{aux-lem-special-case-eq8} and Lemma \ref{brail-gen-inv-thm-tf} that $B_1$ and $B_2$ are geometric progressions of type $(b_1, g, c_1)$ and $(c_1^{-1}, g, c_2)$, respectively for some $b_1, c_1, c_2, g \in G$ with $g \neq 1$. Next we show that for each $a \in A$,
\begin{equation}\label{aux-lem-special-case-eq10}
\mu(a) = \sum_{j=1}^{2}\chi_{B_j}(a).
\end{equation}
Clearly, either $\mu(a) = 1$ or $\mu(a) = 2$. If $\mu(a) = 1$, then $a$ lies in exactly one of the sets $A_1, \ldots, A_m$, and so it lies exactly in one of the sets $B_1$ and $B_2$, and so $\displaystyle\sum_{j=1}^{2}\chi_{B_j}(a) = \mu(a)$ in this case. If $\mu(a) = 2$, then $a \in M$. Hence $a \in B_1 \cap B_2$ which again shows that $\displaystyle\sum_{j=1}^{2}\chi_{B_j}(a) = \mu(a)$. This proves \eqref{aux-lem-special-case-eq10}. Also, it is easy to see that $|B_1| \geq 2$ and $|B_2| \geq 2$. Therefore, the sets $B_1$ and $B_2$ satisfy the conditions of Definition \ref{min-seq}, and hence the sequence $\mathcal{A} = (A_1, \ldots ,A_m)$ is a $(2, g)$-minimizing sequence of sets. 

\noindent \textbf{Claim 2} ($A_r^0$ and $A_s^0$ are geometric progressions of type $(\alpha_r, g, \beta_r)$ and $(\alpha_s, g, \beta_s)$, respectively for some $\alpha_r, \beta_r, \alpha_s, \beta_s, g \in G$ with $g \neq 1$). First, the following identities follows from similar argument used to prove \eqref{aux-lem-special-case-eq8} and \eqref{aux-lem-special-case-eq9}:

\begin{equation}\label{aux-lem-special-case-eq11}
 (A_r^0 \cup A_s^0) C^0 = \Pi^{2}(\mathcal{A})~ \text{and}~ |(A_r^0 \cup A_s^0) C^0| = |(A_r^0 \cup A_s^0)| + |C^0| - 1,
\end{equation}
 
\begin{equation}\label{aux-lem-special-case-eq12}
 A_r^0 (A_s^0 \cup C^0) = \Pi^{2}(\mathcal{A})~ \text{and}~ |A_r^0 (A_s^0 \cup C^0)| = |A_r^0| + |(A_s^0 \cup C^0)| - 1
 \end{equation}
and 
\begin{equation}\label{aux-lem-special-case-eq13}
 A_s^0 (A_r^0 \cup C^0) = \Pi^{2}(\mathcal{A})~ \text{and}~ |A_s^0 (A_r^0 \cup C^0)| = |A_s^0| + |(A_r^0 \cup C^0)| - 1.
\end{equation}
It follows from \eqref{aux-lem-special-case-eq11} and Lemma \ref{aux-lem-inv1} that $A_r^0 \cup A_s^0$, $C^0$ are geometric progressions of types $(a, g, b)$ and $(b^{-1}, g, c)$, respectively for some $a, b, c, g \in G$ with $g \neq 1$, and so
\[A_r^0 \cup A_s^0 = \{ab, agb, \ldots, ag^{u - 1}b\},\]
and
\[C^0 = \{b^{- 1}c, b^{- 1}g c, \ldots, b^{- 1}g^{v - 1}c\},\]
where $u = |A_r^0 \cup A_s^0|$ and $v = |C^0|$. 

It follows from \eqref{aux-lem-special-case-eq12} and Lemma \ref{aux-lem-inv1} that $A_r^0$ and $(A_s^0 \cup C^0)$ are geometric progressions of types $(a_r, g_r, b_r)$ and $(b_r^{-1}, g_r, c_r)$, respectively for some $a_r, b_r, c_r, g_r \in G$ with $g_r \neq 1$. Hence $A_r^0$ is a geometric progression with first term $a_r b_r$ and common ratio $b_r^{-1}g_r b_r$. Since $A_r^0 \subseteq A_r^0 \cup A_s^0$ and $A_r^0$ is a geometric progression, it follows from Lemma \ref{gp-lem4} that there exist integers $p$ and $k$ such that $p \geq 1, k \geq 0$ and 
\[A_r^0 = \{ag^k b, ag^{p + k}b, \ldots, ag^{(n_1 - 1)p + k} b\},\]
where $n_1 = |A_r^0|$. Similarly, it follows from \eqref{aux-lem-special-case-eq13} and Lemma \ref{aux-lem-inv1} that $A_s^0$ and $(A_r^0 \cup C^0)$ are geometric progressions of types $(a_s, g_s, b_s)$ and $(b_s^{-1}, g_s, c_s)$, respectively for some $a_s, b_s, c_s, g_s \in G$ with $g_s \neq 1$. Hence $A_s^0$ is a geometric progression with first term $a_s b_s$ and common ratio $b_s^{-1}g_s b_s$. Since $A_s^0 \subseteq A_r^0 \cup A_s^0$ and $A_s^0$ is a geometric progression, it follows from Lemma \ref{gp-lem4} that there exist integers $q$ and $t$ such that $q \geq 1, t \geq 0$ and 
	\[A_s^0 = \{ag^t b, ag^{q + t}b, \ldots, ag^{(n_2 - 1)q + t} b\},\]
where $n_2 = |A_s^0|$. Without loss of generality, we may assume that $ab \in A_r^0$. Then $k = 0$, and so
\[A_r^0 = \{ab, ag^p b, \ldots, ag^{(n_1 - 1)p}b\}.\]

Since $A_s^0 = \{ag^t b, ag^{q + t}b, \ldots, ag^{(n_2 - 1)q + t} b\}$ is a geometric progression of type $(ag^t, g^q, b)$, it follows from \eqref{aux-lem-special-case-eq13} and Lemma \ref{aux-lem-inv2} that
	\[A_r^0 \cup C^0 = \{b^{- 1}c_1, b^{- 1}g^q c_1, \ldots, b^{- 1}g^{(n_3 - 1) q}c_1\}\]
	for some $c_1 \in G$, where $n_3 = |A_r^0 \cup C^0|$. Since $C^0 \subseteq A_r^0 \cup C^0$ and $b^{- 1}c, b^{- 1}gc \in C^0$, it follows that there exist $i, j \in [0, n_3 - 1]$ such that
	\[b^{- 1}c = b^{- 1}g^{iq}c_1 ~\text{and}~ b^{- 1}gc = b^{- 1}g^{jq}c_1.\]
	Hence 
	\[c = g^{iq}c_1 ~\text{and}~ gc = g^{jq}c_1.\]
	This implies that
	\[g^{(j - i) q - 1} = 1.\]
	Since $g \neq 1$ and $G$ is a torsion-free group, it follows that $(j - i) q = 1$, and so $q = 1$. Hence
	\[A_s^0 = \{ag^t b, ag^{t + 1}b, \ldots, ag^{n_2 + t - 1} b\},\]
	and so $A_s^0$ is geometric progression of type $(ag^t, g, b)$.
	
Since $A_r^0 = \{ab, ag^p b, \ldots, ag^{(n_1 - 1)p}b\}$ is a geometric progression of type $(a, g^p, b)$, it follows from \eqref{aux-lem-special-case-eq12} and Lemma \ref{aux-lem-inv2} that
	\[A_s^0 \cup C^0 = \{b^{- 1}c_2, b^{- 1}g^pc_2, \ldots, b^{- 1}g^{(n_4 - 1) p}c_2\}\]
	for some $c_2 \in G$, where $n_4 = |A_s^0 \cup C^0|$. Since $C^0 \subseteq A_s^0 \cup C^0$ and $b^{- 1}c, b^{- 1}gc \in C^0$, it follows that there exist $i, j \in [0, n_4 - 1]$ such that
	\[b^{- 1}c = b^{- 1}g^{ip}c_2 ~\text{and}~ b^{- 1}gc = b^{- 1}g^{jp}c_2.\]
	Hence 
	\[c = g^{ip}c_2 ~\text{and}~ gc = g^{jp}c_2.\]
This implies that
	\[g^{(j - i) p - 1} = 1.\]
	Since $g \neq 1$ and $G$ is a torsion-free group, it follows that $(j - i) p = 1$, and so $p = 1$.
	Hence
	\[A_r^0 = \{ab, ag b, \ldots, ag^{(n_1 - 1)}b\},\]
	and so $A_r^0$ is geometric progression of type $(a, g, b)$. Therefore, $A_r^0$ and $A_s^0$ are geometric progressions of type $(\alpha_r, g, \beta_r)$ and $(\alpha_s, g, \beta_s)$, respectively, where $\alpha_r = a, \beta_r = b, \alpha_s = ag^t, \beta_s = b$. This proves Claim $2$.

Finally, we prove that if \eqref{aux-lem-special-case-eq} holds, then the sets $A_1 \cup M, \ldots, A_m \cup M$ are geometric progressions of type $(\alpha_1, g, \beta_1), \ldots, (\alpha_m, g, \beta_m)$, respectively for some $\alpha_i, \beta_i, g \in G$ for $i = 1, \ldots, m$ with $g \neq 1$. 

Consider the set $A_i^0$, where $i \in [1, m]\setminus \{r, s\}$. If $A_i = A_r$, then $A_i^0$ is a geometric progression of type $(\alpha_r, g, \beta_r)$. If $A_i \neq A_r$, then we can prove by similar argument that $A_i^0$ and $A_r^0$ are geometric progressions of type $(a_i, g', b_i)$ and $(a_r, g', b_r)$, respectively for some $a_1, b_1, a_r, b_r, g' \in G$ with $g' \neq 1$. Now an application of Lemma \ref{gp-lem2} implies that $A_i^0$ and $A_r^0$ are geometric progressions of type $(\alpha_i, g, \beta_i)$ and $(\alpha_r, g, \beta_r)$, respectively for some $\alpha_i, \beta_i, \alpha_r, \beta_r, g \in G$ with $g \neq 1$. Thus the sets $A_1 \cup M, \ldots, A_m \cup M$ are geometric progressions of type $(\alpha_1, g, \beta_1), \ldots, (\alpha_m, g, \beta_m)$, respectively for some $\alpha_i, \beta_i, g \in G$ for $i = 1, \ldots, m$ with $g \neq 1$. This completes the proof.
\end{proof}

\begin{lemma}\label{aux-lem-dir-tf}
	Let $\ell$ and $m$ be positive integers such that $2 \leq \ell \leq m$. Let $\mathcal{A} = (A_1, \ldots, A_m)$ be a finite sequence of nonempty finite subsets of an arbitrary group $G$, and let $A = A_1 \cup \cdots \cup A_m$. Furthermore, assume that $\mu(a) = \ell$ for some $a \in A$. For each $j \in [1, \ell]$, define $A_j' = \{a \in A : \mu(a) \geq j\}$. Then
\begin{equation}\label{aux-lem-dir-tf-eq1}
  A_1' \cdots A_{\ell}' \subseteq \Pi^{\ell}(\mathcal{A}),
\end{equation}
and
\begin{equation}\label{aux-lem-dir-tf-eq2}
	|A_1'| + \cdots + |A_{\ell}'| =  \sum_{a \in A} \mu (a).
\end{equation}
\end{lemma}

\begin{proof}
Clearly, $A_{\ell}' \subseteq A_{\ell - 1}' \subseteq \cdots \subseteq A_1' ~\text{and}~ A_1' = A$. Let $x' = x_1' \cdots x_{\ell}' \in A_1' \cdots A_{\ell}'$, such that $x_j' \in A_j'$ for each $j \in [1, \ell]$. Then $\mu(x_j') \geq j$ for each $j \in [1, \ell]$. Hence for each $j \in [1, \ell]$, $x_j'$ belongs to at least $j$ distinct sets among the sets $A_1, \ldots, A_m$. Let $x_1' \in A_{i_1}$. Since $\mu(x_2') \geq 2$, it follows that there exists a set $A_{i_2}$ distinct from $A_{i_1}$ such that $x_2' \in A_{i_2}$. Suppose we have chosen the distinct sets $A_{i_1}, \ldots, A_{i_{j-1}}$ such that $x_1' \in A_{i_1}, \ldots, x_{j-1}' \in A_{i_{j-1}}$, respectively. Now since $\mu(x_j') \geq j$, it follows that $x_j'$ belongs to at least $j$ distinct sets among the sets $A_1, \ldots, A_m$. Therefore, there exists a set $A_{i_j}$ distinct from the sets $A_{i_1}, \ldots, A_{i_{j-1}}$ such that $x_j' \in A_{i_j}$. Thus we can find distinct sets $A_{i_1}, \ldots, A_{i_{\ell}}$ such that $x_j' \in A_{i_j}$ for $j = 1, \ldots, \ell$. Hence
    \[x' = x_1' \cdots x_{\ell}' \in A_{i_1} \cdots  A_{i_{\ell}} \subseteq \Pi^{\ell}(\mathcal{A}).\]
    Since $x'$ is an arbitrary element of $A_1' \cdots A_{\ell}'$, it follows that 
    \[A_1' \cdots A_{\ell}' \subseteq \Pi^{\ell}(\mathcal{A}).\]
This proves \eqref{aux-lem-dir-tf-eq1}.
    
To prove \eqref{aux-lem-dir-tf-eq2}, first note that $A = A_1' \cup \cdots \cup A_{\ell}'$ and each $a \in A$ lies only in the sets $A_1', \ldots, A_{\mu(a)}'$ among the sets $A_1', \ldots, A_{\ell}'$. Hence
    \[\mu (a) = \sum_{j = 1}^{\ell} \chi_{A_j'}(a)\]
    for each $g \in A$, and so
    \begin{align*}
    	 \sum_{a \in A} \mu (a) = \sum_{a \in A} \sum_{j = 1}^{\ell} \chi_{A_j'}(a) = \sum_{j = 1}^{\ell} \sum_{a \in A} \chi_{A_j'}(a) = \sum_{j = 1}^{\ell} \sum_{a \in A = A_1' \cup \cdots \cup A_{\ell}'} \chi_{A_j'}(a) = \sum_{j = 1}^{\ell} |A_j'|.
    \end{align*}
    This proves \eqref{aux-lem-dir-tf-eq2} and completes the proof.
\end{proof}

\subsection{Proof of main theorems and corollary} \label{subsec-main-thm-torsionfree}

\begin{proof}[Proof of Theorem \ref{productset-tf-dir-thm}]
	It follows from Lemma \ref{aux-lem-existance} that there exist sets $A_1^2, \ldots, A_{\ell}^2 \subseteq A$ such that
	\[A_1^2 \cdots A_{\ell}^2 \subseteq \Pi^{\ell}(\mathcal{A}),\]
	and
	\[|A_1^2| + \cdots + |A_{\ell}^2| =  \sum_{a \in A} \mu (a).\]
Therefore, it follows form Lemma \ref{kemp-gen-thm-tf} that
	\begin{align*}
		|\Pi^{\ell}(\mathcal{A})| &\geq |A_1^2 \cdots A_{\ell}^2| \geq |A_1^2| + \cdots + |A_{\ell}^2| - \ell + 1 = \sum_{a \in A} \mu (a) - \ell + 1.
	\end{align*}
This proves the first part of the theorem. The best possibility of lower bound in \eqref{productset-tf-dir-thm-eq1} is shown in Appendix \ref{best-lower-bound}. 
\end{proof}

\begin{proof}[Proof of Theorem \ref{productset-tf-inv-thm1}.]
If $\mathcal{A}$ is a $(\ell, g)$-minimizing sequence of sets, then it follows from Lemma \ref{aux-lem-inv3} that
\[|\Pi^{\ell}(\mathcal{A})| =  \sum_{a \in A} \mu (a) - \ell + 1.\]
Now assume that $|\Sigma^{\ell}(\mathcal{A})| =  \sum_{a \in A} \mu (a) - \ell + 1$. It follows from Lemma \ref{aux-lem-existance} that there exist sets $A_1^2, \ldots, A_{\ell}^2 \subseteq A$ satisfying such that
\[|A_j^2| \geq |A_j| \geq 2~\text{for each}~ j \in [1, \ell],\]
\[ \mu (a) = \sum_{j = 1}^{\ell} \chi_{A_j^2}(a)~\text{for each}~ a \in A,\]
	\[A_1^2 \cdots A_{\ell}^2 \subseteq \Pi^{\ell}(\mathcal{A}),\]
	and
	\[|A_1^2| + \cdots + |A_{\ell}^2| =  \sum_{a \in A} \mu (a).\]
Therefore, it follows form Lemma \ref{kemp-gen-thm-tf} that
	\begin{align*}
		\sum_{a \in A} \mu (a) - \ell + 1 = |\Pi^{\ell}(\mathcal{A})| &\geq |A_1^2 \cdots A_{\ell}^2| \\
& \geq |A_1^2| + \cdots + |A_{\ell}^2| - \ell + 1 = \sum_{a \in A} \mu (a) - \ell + 1.
	\end{align*}
Therefore,
\[A_1^2 \cdots A_{\ell}^2 = \Pi^{\ell}(\mathcal{A}),\]
and
\[|A_1^2 \cdots A_{\ell}^2| = |A_1^2| + \cdots + |A_{\ell}^2| - \ell + 1.\]
Hence, it follows from Lemma \ref{brail-gen-inv-thm-tf} that for each $j \in [1, \ell]$, the set $A_j^2$ is a geometric progressions of type $(\alpha_j, g, \beta_j)$ for some $\alpha_j, \beta_j, g \in G$ with $g \neq 1$ for each $j \in [1, \ell]$, where $\alpha_{j + 1} = \beta_j^{- 1}$ for each $j \in [1, \ell - 1]$. Thus the sets $A_1^2, \ldots, A_{\ell}^2$ satisfy the conditions $(1)-(4)$ of Definition \ref{min-seq}. Therefore, the sequence $\mathcal{A}$ is a $(\ell, g)$- minimizing sequence of sets. Furthermore, it follows from Lemma \ref{aux-lem-inv-tf} that the sets $A_1 \cup M, \ldots, A_m \cup M$ are geometric progressions of type $(\alpha_1, g, \beta_1), \ldots, (\alpha_m, g, \beta_m)$, respectively for some $\alpha_i, \beta_i, g \in G$ for $i = 1, \ldots, m$ with $g \neq 1$. This completes the proof.
\end{proof} 

\begin{proof}[Proof of Corollary \ref{productset-tf-inv-thm1-cor}]
It follows from Theorem \ref{productset-tf-inv-thm1} that $|\Pi^{\ell}(\mathcal{A})| =  \sum_{a \in A} \mu (a) - \ell + 1$ if and only if $\mathcal{A}$ is a $(\ell, g)$-minimizing sequence of sets. Also if the equality \eqref{productset-tf-inv-thm1-cor-eq1} holds, then the sets $A_1 \cup M, \ldots, A_m \cup M$ are geometric progressions with common ratio $g$ for some $g \in G$ with $g \neq 1$. Now we show that $A$ is also a geometric progression with the same common ratio $g$. Now we consider the following cases:

\noindent \textbf{Case 1} ($\ell = 2$). Consider the sets $A_1$ and $A_2$. Let $A_1^0 = A_1 \cup M$, $A_2^0 = A_2 \cup M$ and $C = A_3 \cup \cdots \cup A_m$. Let $C^0 = C \cup M$, and let $B_1 = A_1^0 \cup C^0$ and $B_2 = A_2^0$. Then it follows from the same argument (by taking $r = 1$ and $s = 2$) as given in the proof of Claim $1$ in the proof of Lemma \ref{aux-lem-special-case} that the sets $B_1 = A_1^0 \cup C^0 = A_1 \cup A_3 \cup \cdots \cup A_m \cup M$ and $B_2 = A_2^0 = A_2 \cup M$ satisfy the conditions $(1) - (4)$ of Definition \ref{min-seq}. Therefore, the sets $B_1$ and $B_2$ are geometric progression with the same common ratio $g$. Since $B_1 \cap B_2 \neq \emptyset$, it follows from  Lemma \ref{gp-lem5} that $A = B_1 \cup B_2$ is also a geometric progression with the same common ratio $g$.\\

\noindent \textbf{Case 2} ($\ell \geq 3$). 
It follows from the proof of Claim $1$ and Claim $5$ in Lemma \ref{aux-lem-inv-tf} that the sets $A_1^2, \ldots, A_{\ell - 2}^2, A_{\ell}^0 \cup C, A_{\ell - 1}^0, A_{\ell}^0$ are geometric progressions with the same common ratio $g$. Hence $A_1^2 = A_1 \cup B_1 \cup M$, $A_{\ell - 1}^0 = A_{\ell - 1} \cup M$ and $A_{\ell}^0 = A_{\ell} \cup M$ are geometric progressions with the same common ratio $g$, where $B_1 = A_{\ell + 1} \cup \cdots \cup A_m$. Since $\Pi^{\ell}(\mathcal{A})$ is invariant under reordering of the sets $A_1, A_2, \ldots, A_m$, it follows that the sets $A_1 \cup M, \ldots, A_m \cup M$, $A_1 \cup B_1 \cup M$, $A_2 \cup B_1 \cup M, \ldots, A_{\ell} \cup B_1\cup M$ are geometric progressions with the same common ratio $g$. Now by applying Lemma \ref{gp-lem5} on the sets $A_1 \cup B_1 \cup M$ and $A_2 \cup B_1\cup M$, we see that their union $A_1 \cup A_2 \cup B_1 \cup M$ is the geometric progressions with the same common ratio $g$. Again, by applying Lemma \ref{gp-lem5} on the sets $A_1 \cup A_2 \cup B_1 \cup M$ and $A_3 \cup B_1\cup M$, we see that their union $A_1 \cup A_2 \cup A_3 \cup B_1 \cup M$ is the geometric progressions with the same common ratio $g$. By repeating this process, we get that $A_1 \cup \cdots \cup A_{\ell} \cup B_1 \cup M$ is the geometric progressions with the same common ratio $g$. Since $A_1 \cup \cdots \cup A_{\ell} \cup B_1 \cup M = A$, the proof follows.
\end{proof}

\begin{proof}[Proof of Theorem \ref{productset-no-min-sq-ext-thm}]
	It follows from the proof of Corollary  \ref{productset-tf-inv-thm1-cor} that the sets $A, A_1 \cup M, \ldots, A_m \cup M$, $A_1 \cup B_1 \cup M$, $A_2 \cup B_1 \cup M, \ldots, A_{\ell} \cup B_1\cup M$ are geometric progressions with the same common ratio $g$ with $g \neq 1$, where $B_1 = A_{\ell + 1} \cup \cdots \cup A_m$. For each $j \in [1, \ell]$, let $|A_j \cup B_1 \cup M| = n_j$. Let
\[A_1 \cup B_1 \cup M = \{\alpha g^r : r \in [0, n_1 - 1]\},\]
and for each $j \in [2, \ell]$, let
\[A_j \cup B_1 \cup M = \{\alpha_j g^r : r \in [0, n_j - 1]\}.\]

For each $j \in [1, m]$, let $|A_j \cup M| = k_j$. Since for each $j \in \{1\} \cup [\ell + 1, m]$, the set $A_j \cup M $ is subset of $A_j \cup M \subseteq A_1 \cup B_1 \cup M$, and $A_j \cup M$ is also a geometric progression with the same common ratio $g$, it follows that there exist integers $r_j$ such that $r_j, k_j \in [0, r - 1]$ and
\[A_j \cup M = \{\alpha g^r : r \in [r_j, k_j + r_j - 1]\}.\]

Since $\alpha g^{r_{\ell + 1}} \in A_{\ell + 1} \cup M$ and $A_{\ell + 1} \cup M \subseteq A_j \cup B_1 \cup M$ for each $j \in [2, \ell]$, it follows that there exists an integer $t_j \in [0, n_j - 1]$ such that
\[\alpha g^{r_{\ell + 1}} = \alpha_j g^{t_j}\]
which implies that
\[\alpha_j = \alpha g^{r_{\ell + 1} - t_j}\]
for each $j \in [2, \ell]$. Therefore,
\[A_j \cup B_1 \cup M = \{\alpha_j g^r : r \in [r_{\ell + 1} - t_j, r_{\ell + 1} - t_j + n_j - 1]\}\]
for each $j \in [2, \ell]$. Now, since for each $j \in [2, \ell]$, the set $A_j \cup M$ is a subset of $A_j \cup B_1 \cup M$ and $A_j \cup M$ is also the geometric progression with the same common ratio $g$, it follows that there exist integers $i_j$ such that
\[A_j \cup M = \{\alpha g^r : r \in [r_j, k_j + r_j - 1]\}.\]
Without loss of generality, we may assume that
    \[r_1 \leq r_2 \leq \cdots \leq r_m.\]
Now if $r_j  = r_{j + 1}$ for some $j \in [1, m - 1]$, then $\alpha g^{r_j}$ and $\alpha g^{r_j + 1}$ both belong to $(A_j \cup M) \cap (A_{j + 1} \cup M)$, which contradicts the fact that $|\{a \in A : \mu (a) \geq 2 \}| \leq 1$. Hence
    \begin{equation}\label{productset-no-min-sq-ext-thm-eq1}
    	r_1 < r_2 < \cdots < r_m.
    \end{equation} 
Furthermore, if $r_{j + 1} \leq k_j + r_j - 2$ for some $j \in [1, m - 1]$, then $\alpha g^{r_{j + 1}}$ and $\alpha g^{r_{j + 1} + 1 }$ both belong to $(A_j \cup M) \cap (A_{j + 1} \cup M)$, which again contradicts the fact that $|\{a \in A : \mu (a) \geq 2 \}| \leq 1$. Hence 
    \begin{equation}\label{productset-no-min-sq-ext-thm-eq2}
    	k_j + r_j - 1 \leq r_{j + 1}
    \end{equation}
    for each $j \in [1, m - 1]$. Therefore, it follows from \eqref{productset-no-min-sq-ext-thm-eq1} and \eqref{productset-no-min-sq-ext-thm-eq2} that
\begin{equation}\label{productset-no-min-sq-ext-thm-eq3}
  [r_2, k _2 + r_2 - 1] \subseteq [r_1, k_{\ell + 1} + r_{\ell + 1} - 1].
\end{equation}
 
Since all the sets $A_j \cup M$ for $j \in \{1\} \cup [\ell + 1, m]$ are the geometric progressions with the same common ratio $g$, and their union $A_1 \cup B_1 \cup M$ is also the geometric progressions with the same common ratio $g$, it follows from \eqref{productset-no-min-sq-ext-thm-eq1} and \eqref{productset-no-min-sq-ext-thm-eq2} that
  \[~\text{either}~ k_1 + r_1 = r_{\ell + 1} ~\text{or}~ k_1 + r_1 + 1 = r_{\ell + 1},\]
and so
\[(A_1 \cup M) \cup (A_{\ell + 1} \cup M) = \{\alpha g^r: r \in [r_1, k_{\ell + 1} + r_{\ell + 1} - 1]\}.\]
Hence it follows from \eqref{productset-no-min-sq-ext-thm-eq3} that
    \[A_2 \cup M \subseteq (A_1 \cup M) \cup (A_{\ell + 1} \cup M) = A_1 \cup A_{\ell + 1} \cup M\]
which again contradicts the fact that $|\{a \in A : \mu (a) \geq 2\}| \leq 1$. Therefore, no sequence $\mathcal{A}$ satisfying the assumption of the lemma exists. This completes the proof.  
\end{proof}

\begin{proof}[Proof of Theorem \ref{productset-tf-inv-thm2}]
It follows from \eqref{productset-tf-inv-thm2-eq1} and Lemma \ref{aux-lem-dir-tf} that 
	\begin{align*}
		\sum_{i = 1}^{\ell} |A_i'| - \ell + 1 &= |\Pi^{\ell}(\mathcal{A})| \geq |A_1' + \cdots + A_{\ell}'| \geq \sum_{i = 1}^{\ell} |A_i'| - \ell + 1, 
	\end{align*} 
which implies that
	\[|A_1' \cdots A_{\ell}'| = \sum_{i = 1}^{\ell} |A_i'| - \ell + 1.\]	
Since $|\{a \in A : \mu (a) \geq 2\}| \geq 2$, it follows that $2 \leq k \leq \ell$. Since $A_{\ell}' \subseteq \cdots \subseteq A_1'$, it follows that 
\[|A_1'| \geq \cdots \geq |A_k'| \geq 2\]
and 
\[|A_{k + 1}'| = \cdots = |A_{\ell}'| = 1.\]
Hence  
	\begin{align*}
		|A_1' \cdots A_k'| = |A_1' \cdots A_{\ell}'| &= \sum_{i = 1}^{\ell} |A_i'| - \ell + 1 \\
&= \sum_{i = 1}^k |A_i'| + (\ell - k)- \ell + 1 = \sum_{i = 1}^k |A_i'| - k + 1.
	\end{align*}
	Therefore, it follows from Lemma \ref{brail-gen-inv-thm-tf} that for each $j \in [1, k]$, the set $A_j'$ is a geometric progressions of type $(\alpha_j, g, \beta_j)$ for some $\alpha_j, \beta_j, g \in G$  with $g \neq 1$, where $\alpha_{j + 1} = \beta_j^{- 1}$ for each $j \in [1, k-1]$.
 This completes the proof.
\end{proof}

\section{Generalized sumset $\Sigma^{\ell}(\mathcal{A})$ in $\mathbb{Z}_p$} \label{sec-prime-order-group}

The following theorem is a special case of Theorem \ref{dgm-thm}. Here we give an independent proof of this result.

\begin{theorem}\label{zp-dir-thm}
	Let $\ell$ and $m$ be integers such that $2 \leq \ell \leq m$. Let $\mathcal{A} = (A_1, \ldots, A_m)$ be a finite sequence of nonempty subsets of $\mathbb{Z}_p$, where $p$ is a prime. Then
	\begin{equation}\label{zp-dir-thm-eq1}
		|\Sigma^{\ell}(\mathcal{A})| \geq \min \biggl(p,  \sum_{a \in A} \mu (a) - \ell + 1\biggr).
	\end{equation}
	The lower bond in \eqref{zp-dir-thm-eq1} is best possible.
\end{theorem}

\begin{proof}
	It follows from Lemma \ref{aux-lem-existance} that there exist sets $A_1^2, \ldots, A_{\ell}^2 \subseteq A$ such that
	\[A_1^2 + \cdots + A_{\ell}^2 \subseteq \Sigma^{\ell}(\mathcal{A}),\]
	and
	\[|A_1^2| + \cdots + |A_{\ell}^2| =  \sum_{a \in A} \mu (a).\]
Therefore, it follows form Theorem \ref{cauchy-davenport-thm} that
	\begin{align*}
		|\Sigma^{\ell}(\mathcal{A})| &\geq |A_1^2 + \cdots + A_{\ell}^2| \geq  \min \biggl(p,  \sum_{a \in A} \mu (a) - \ell + 1\biggr).
	\end{align*}
This proves the first part of the theorem. The best possibility of lower bound in \eqref{zp-dir-thm-eq1} can be shown by the method similar to the method in Appendix \ref{best-lower-bound}. 
\end{proof}

The next theorem characterizes the extremal sets in $\mathbb{Z}_p$ for the sumset $\Sigma^{\ell}(\mathcal{A})$.

\begin{theorem}\label{zp-inv-thm}
	Let $\ell$ and $m$ be positive integers such that $2 \leq \ell < m$. Let $\mathcal{A} = (A_1, \ldots, A_m)$ be a sequence of finite subsets of $\mathbb{Z}_p$, where $p$ is prime and $|A_i|  \geq 2$ for each $i  \in [1, m]$. Let $A = A_1 \cup \cdots \cup A_m$, and let $M = \{a \in A: \mu(a) = \ell\}$. Let
	\begin{equation}\label{zp-inv-thm-eq1}
		|\Sigma^{\ell}(\mathcal{A})| <
		\begin{cases}
			p - 1, & \mbox{if } \ell = 2; \\
			p, & \mbox{if } \ell \geq 3.
		\end{cases}
	\end{equation}
Then
	\begin{equation}\label{zp-inv-thm-eq2}
		|\Sigma^{\ell}(\mathcal{A})| =  \sum_{a \in A} \mu (a) - \ell + 1,
	\end{equation}
	if and only if $\mathcal{A}$ is a $(\ell, d)$-minimizing sequence of sets. Furthermore, if \eqref{zp-inv-thm-eq1} and \eqref{zp-inv-thm-eq2} hold, then the following conclusions hold:
	\begin{enumerate}
		\item If $\ell = 2$, then $A_1 \cup M, \ldots, A_m \cup M$ are arithmetic progressions.
		\item If $\ell \geq 3$, then the sets $A, A_1 \cup M, \ldots, A_m \cup M$ are arithmetic progressions with the same common difference $d$. 
       \item Furthermore, if $\ell \geq 3$, then the set $A$ is also the arithmetic progressions with the common difference $d$. 
	\end{enumerate}
\end{theorem}

For the proof of Theorem \ref{zp-inv-thm}, we need the following lemmas.

\begin{lemma}\label{vosper-gen-inv-thm1}
	Let $\ell \geq 2$ be an integer. Let $A_1, \ldots, A_{\ell}$ be nonempmty subsets of $\mathbb{Z}_p$ such that $|A_i| \geq 2$ for each  $i \in [1, \ell]$. Let $|A_1| + \cdots + |A_{\ell}| - \ell  + 1 \leq p - 2$. Then
	\[|A_1 + \cdots + A_{\ell}| = |A_1| + \cdots + |A_{\ell}| - \ell  + 1\]
	if and only if $A_1, \ldots, A_{\ell}$ are arithmetic progressions with the same common difference.
\end{lemma}

\begin{proof}
	If $A_1, \ldots, A_{\ell}$ are arithmetic progressions with the same common difference, then it is easy to verify that
	\[|A_1 + \cdots + A_{\ell}| = |A_1| + \cdots + |A_{\ell}| - \ell  + 1 \leq p - 2.\]
	Now we assume that $|A_1 + \cdots + A_{\ell}| = |A_1| + \cdots + |A_{\ell}| - \ell  + 1 \leq p - 2$. If $\ell = 2$, then it follows from Theorem \ref{vosperthm} that $A_1$ are $A_2$ are arithmetic progressions with same common difference. Now we assume that the result holds for $\ell - 1$ sets, where $\ell \geq 3$. Let $A = A_1 + \cdots + A_{\ell - 1}$. Since $A + A_{\ell} = A_1 + \cdots + A_{\ell} \neq \mathbb{Z}_p$, it follows that $A = A_1 + \cdots + A_{\ell - 1} \neq \mathbb{Z}_p$. Therefore, it follows from Theorem \ref{cauchy-davenport-thm} that
	\begin{align}\label{vosper-gen-inv-thm1-eq1}
		|A_1| + \cdots + |A_{\ell}| - \ell + 1 = |A_1 + \cdots + A_{\ell}| &= |A + A_{\ell}| \notag \\
        & \geq |A| + |A_{\ell}| - 1 \notag\\
		&= |A_1 + \cdots + A_{\ell - 1}| + |A_{\ell}| - 1 \notag\\
		&\geq (|A_1| + \cdots + |A_{\ell - 1}| - \ell + 2) + |A_{\ell}| - 1 \notag\\
		&= |A_1| + \cdots + |A_{\ell}| - \ell + 1.
	\end{align}
	Hence
	\[|A_1 + \cdots + A_{\ell - 1}| = |A_1| + \cdots + |A_{\ell - 1}| - \ell + 2 \leq p - 2.\]
Therefore, it follows from the induction hypothesis that $A_1, \ldots, A_{\ell - 1}$ are arithmetic progressions with the same common difference, say $d$. Therefore, $A = A_1 + \cdots + A_{\ell - 1}$ is also an arithmetic progression with the common difference $d$. Now, it also follows from \eqref{vosper-gen-inv-thm1-eq1} that
	\[|A + A_{\ell}| = |A| + |A_{\ell}| - 1 \leq p -2.\]
Therefore, it follows from Theorem \ref{vosperthm} that $A$ and $A_{\ell}$ are arithmetic progressions with the same common difference. Since $A$ is an arithmetic progression with common difference $d$, it follows that $A_{\ell}$ is also an arithmetic progression with common difference $d$. Thus $A_1, \ldots, A_{\ell}$ are arithmetic progressions with the same common difference.
\end{proof}

\begin{lemma}\label{vosper-gen-inv-thm2}
	Let $\ell \geq 3$ be a positive integer. Let $A_1, \ldots, A_{\ell}$ be nonempty subsets of $\mathbb{Z}_p$ such that $|A_i| \geq 2$ for each  $i \in [1, \ell]$. Let $A_1 + \cdots + A_{\ell} \neq \mathbb{Z}_p$. Then
	\[|A_1 + \cdots + A_{\ell}| = |A_1| + \cdots + |A_{\ell}| - \ell  + 1\]
	if and only if $A_1, \ldots, A_{\ell}$ are arithmetic progressions with the same common difference.	
\end{lemma}

\begin{proof}
	If $A_1, \ldots, A_{\ell}$ are arithmetic progressions with the same common difference, then it is easy to verify
	\[|A_1 + \cdots + A_{\ell}| = |A_1| + \cdots + |A_{\ell}| - \ell  + 1.\]
	Now we assume that $|A_1 + \cdots + A_{\ell}| = |A_1| + \cdots + |A_{\ell}| - \ell  + 1 \leq p - 1$. If $|A_1 + \cdots + A_{\ell}| = |A_1| + \cdots + |A_{\ell}| - \ell  + 1 < p - 1$, then the result follows from Lemma \ref{vosper-gen-inv-thm1}. Now we assume that $|A_1 + \cdots + A_{\ell}| = |A_1| + \cdots + |A_{\ell}| - \ell  + 1 = p - 1$. By the similar argument as in the proof of Lemma \ref{vosper-gen-inv-thm1} we show that 
	\[|A_1 + \cdots + A_{\ell - 1}| = |A_1| + \cdots + |A_{\ell - 1}| - \ell  + 2 < p - 1 .\] 
	Since $\ell \geq 3$ and $|A_1 + \cdots + A_{\ell - 1}| = |A_1| + \cdots + |A_{\ell - 1}| - \ell  + 2 < p - 1$, it follows from lemma \ref{vosper-gen-inv-thm1} that $A_1, \ldots, A_{\ell - 1}$ are arithmetic progressions with the same common difference. Similarly, we show that  $A_2, \ldots, A_{\ell}$ are arithmetic progressions with the same common difference. Therefore, $A_1, \ldots, A_{\ell}$ are arithmetic progressions with the same common difference.
\end{proof}

\begin{lemma}\label{zp-inv-lem}
	Let $\ell, m$ and $p$ be positive integers such that $2 \leq \ell < m$ and $p$ is prime. Let $\mathcal{A} = (A_1, \ldots, A_m)$ be a finite sequence of nonempty subsets of $\mathbb{Z}_p$, where $|A_i|  \geq 2$ for each $i  \in [1, m]$. Let $A = A_1 \cup \cdots \cup A_m$, and let $M = \{a \in A: \mu(a) = \ell\}$. Let
	\begin{equation*}
		|\Sigma^{\ell}(\mathcal{A})| =  \sum_{a \in A} \mu (a) - \ell + 1 <
		\begin{cases}
			p - 1, & \mbox{if } \ell = 2; \\
			p, & \mbox{if } \ell \geq 3.
		\end{cases}
	\end{equation*}
Then the following conclusions hold:
	\begin{enumerate}
		\item If $\ell = 2$, then $A_1 \cup M, \ldots, A_m \cup M$ are arithmetic progressions.
		\item If $\ell \geq 3$, then the sets $A, A_1 \cup M, \ldots, A_m \cup M$ are arithmetic progressions with the same common difference. 
	\end{enumerate}
\end{lemma}

\begin{proof}
	We consider two cases:
	
	\noindent \textbf{Case 1} ($\ell = 2$). Consider any two sets $A_r$ and $A_s$, where $r, s \in [1, m]$. Let $A_r^0 = A_r \cup M$, $A_s^0 = A_s \cup M$ and $C = \displaystyle\bigcup_{k \in [1, m] \setminus \{r, s\}}A_k$. Let $C^0 = C \cup M$, and let $B_1 = A_r^0 \cup C^0$ and $B_2 = A_s^0$. Then it follows from the same argument (adapted with binary operation of addition) as given in the proof of Claim $1$ in the proof of Lemma \ref{aux-lem-special-case} that $B_2 = A_s^0 = A_s \cup M$ is an arithmetic progression, where we use Theorem \ref{cauchy-davenport-thm} and Theorem \ref{vosperthm} in place of Lemma \ref{kemp-gen-thm-tf} and Lemma \ref{brail-gen-inv-thm-tf}, respectively. Since $s \in [1, m]$ is arbitrary, it follows that the sets $A_1 \cup M, \ldots, A_m \cup M$, $A_1 \cup B_1 \cup M$, $A_2 \cup B_1 \cup M, \ldots, A_{\ell} \cup B_1\cup M$ are arithmetic progressions.\\
	
\noindent \textbf{Case 2} ($\ell \geq 3$). 
	First we show that $A_1 \cup M, \ldots, A_m \cup M$ are arithmetic progressions with the same common difference $d$ for some $d \in \mathbb{Z}_p$, with $d \neq 0$. The proof of this fact follows from the similar argument (adapted for $\mathbb{Z}_p$ with respect to the binary operation of addition) as in the proof of Lemma \ref{aux-lem-inv-tf}, where we use Theorem \ref{cauchy-davenport-thm} and Lemma \ref{vosper-gen-inv-thm2} in place of Lemma \ref{kemp-gen-thm-tf} and Lemma \ref{brail-gen-inv-thm-tf}, respectively. Furthermore, a similar modification in the argument given in Case $2$ in the proof of Corollary \ref{productset-tf-inv-thm1-cor} shows that the set $A$ is also the arithmetic progression with common difference $d$, where we use the Lemma \ref{gp-lem5} adapted with respect to the binary operation of addition. 
\end{proof}

\begin{proof}[Proof of Theorem \ref{zp-inv-thm}]
	The proof of the first part of the theorem follows from the argument similar to that in proof of Theorem  \ref{productset-tf-inv-thm1}, where we use Theorem \ref{cauchy-davenport-thm} in place of Lemma \ref{kemp-gen-thm-tf}, and we use Lemma \ref{vosper-gen-inv-thm1} and Lemma \ref{vosper-gen-inv-thm2} in place of Lemma \ref{brail-gen-inv-thm-tf}. The last part of the theorem follows from Lemma \ref{zp-inv-lem}. 
\end{proof}

\section{$\Sigma^{\ell}(\mathcal{A})$ in arbitrary groups} \label{sec-abelian-groups}
In this section, we assume that the group $G$ is written additively. Let $p(G)$ denote the order of the smallest nontrivial subgroup of a group $G$, or $\infty$ if no such subgroups exist. K\'arolyi \cite{karolyi2005} proved the following result in arbitrary finite groups.

\begin{theorem}\label{karolyi-thm}
	Let $A$ and $B$ be nonempty finite subsets of a finite group $G$. Then 
	\[|A + B| \geq \min (p(G), |A| + |B| - 1).\] 
\end{theorem}

DeVos \cite{devos2016} extended this result to arbitrary groups (not necessarily finite) by proving the following result.

\begin{theorem}\label{devos-extension}
	Let $A$ and $B$ be nonempty finite subsets of a group $G$. Then 
	\[|A + B| \geq \min (p(G), |A| + |B| - 1).\] 
\end{theorem}

An easy generalization of the above theorem is the following result.

\begin{theorem}\label{devos-extension-reg-sum}
	Let $\ell \geq 2$ be an integer. Let $A_1, \ldots, A_{\ell}$ be nonempty finite subsets of a group $G$ with $|A_i| \geq 2$ for each $i \in [1, \ell]$. Then 
	\[|A_1 + \cdots + A_{\ell}| \geq \min (p(G), |A| + \cdots + |A_{\ell}| - \ell + 1).\] 
\end{theorem}

Kemperman proved the following result.

\begin{theorem}[\cite{kemperman1960}, Corollary]\label{kemp-inv-thm-abl}
	Let $A$ and $B$ be nonempty finite subsets of an abelian  group $G$. Let
	\[|A + B| = |A| + |B| - 1 \leq p(G) - 2.\]  
	Then $A, B$ are arithmetic progressions with common difference. 
\end{theorem}

The following generalizations of above theorem follows from similar arguments used in the proof of Lemma \ref{vosper-gen-inv-thm1} and Lemma \ref{vosper-gen-inv-thm2}.

\begin{lemma}\label{kemp-inv-thm-abl-lem1}
	Let $\ell \geq 2$ be a positive integer. Let $A_1, \ldots, A_{\ell}$ be nonempty subsets of an abelian group $G$ such that $|A_i| \geq 2$ for each  $i \in [1, \ell]$. Let
	\[|A_1 + \cdots + A_{\ell}| = |A_1| + \cdots + |A_{\ell}| - \ell  + 1 \leq p(G) - 2.\]
	Then $A_1, \ldots, A_{\ell}$ are arithmetic progressions with the same common difference.	
\end{lemma}

\begin{lemma}\label{kemp-inv-thm-abl-lem2}
	Let $\ell \geq 3$ be an integer. Let $A_1, \ldots, A_{\ell}$ be nonempty subsets of an abelian group $G$ such that $|A_i| \geq 2$ for each  $i \in [1, \ell]$. Let
	\[|A_1 + \cdots + A_{\ell}| = |A_1| + \cdots + |A_{\ell}| - \ell  + 1 < p(G).\]
	Then $A_1, \ldots, A_{\ell}$ are arithmetic progressions with the same common difference.	
\end{lemma}

Using the above results, we can prove the following theorems by similar arguments used in the proof of Theorem \ref{zp-dir-thm} and Theorem \ref{zp-inv-thm} in the cyclic group $\mathbb{Z}_p$.

\begin{theorem}\label{sigma-abl-dir-thm}
	Let $\ell$ and $m$ be positive integers such that $\ell \leq m$. Let $\mathcal{A} = (A_1, \ldots, A_m)$ be a sequence of finite subsets of a group $G$, and let $A = A_1 \cup \cdots \cup A_m$. Then
	\begin{equation}\label{sigma-abl-dir-thm-eq1}
		|\Sigma^{\ell}(\mathcal{A})| \geq \min \biggl(p(G),  \sum_{a \in A} \mu (a) - \ell + 1\biggr).
	\end{equation}
	The lower bond in \eqref{sigma-abl-dir-thm-eq1} is best possible.
\end{theorem}

The following theorem is an inverse theorem which characterizes the extremal sets.
\begin{theorem}\label{sigma-abl-inv-thm}
	Let $\ell$ and $m$ be integers such that $2 \leq \ell < m$. Let $\mathcal{A} = (A_1, \ldots, A_m)$ be a sequence of finite subsets of an abelian group $G$, and let $A = A_1 \cup \cdots \cup A_m$. Let
	\begin{equation*}
		|\Sigma^{\ell}(\mathcal{A})| <
		\begin{cases}
			p(G) - 1, & \mbox{if } \ell = 2; \\
			p(G), & \mbox{if } \ell \geq 3.
		\end{cases}
	\end{equation*}
	Furthermore, assume that $|A_i|  \geq 2$ for each $i  \in [1, m]$, and let $M = \{a \in A: \mu(a) = \ell\}$. Then
	\begin{equation}\label{sigma-abl-inv-thm-eq1}
		|\Sigma^{\ell}(\mathcal{A})| =  \sum_{a \in A} \mu (a) - \ell + 1,
	\end{equation}
	if and only if $\mathcal{A}$ is a $(\ell, d)$ minimizing sequence of sets. Furthermore, if \eqref{sigma-abl-inv-thm-eq1} holds, then the following conclusions hold:
	\begin{enumerate}
		\item If $\ell = 2$, then $A_1 \cup M, \ldots, A_m \cup M$ are arithmetic progressions.
		\item If $\ell \geq 3$, then the sets $A, A_1 \cup M, \ldots, A_m \cup M$ are arithmetic progressions with the same common difference. 
	\end{enumerate}
\end{theorem}

The following theorem in abelian groups corresponds to Theorem \ref{productset-tf-inv-thm2} and its proof is similar to the proof that theorem.
\begin{theorem}\label{sigma-abl-inv-thm1}
	Let $\ell$ and $m$ be integers such that $2 \leq \ell < m$. Let $\mathcal{A} = (A_1, \ldots, A_m)$ be a sequence of finite subsets of an abelian group $G$, and let $A = A_1 \cup \cdots \cup A_m$. Let $\mu (a) = \ell$ for some $a \in A$. Furthermore, assume that $|\{a \in A : \mu(a) \geq 2\}| \geq 2$. Let
	\begin{equation*}
		|\Sigma^{\ell}(\mathcal{A})| =  \sum_{a \in A} \mu (a) - \ell + 1 <
		\begin{cases}
			p(G) - 1, & \mbox{if } \ell = 2; \\
			p(G), & \mbox{if } \ell \geq 3,
		\end{cases}
	\end{equation*}
	and
	\[A_j' = \{a \in A : \mu (a) \geq j\}\]
	for each $j \in [1, \ell]$, and let $k$ be the largest integer such that $|A_j'| \geq 2$ for each $j \in [1, k]$. Then  for each $j \in [1, k]$, the set $A_1' = A, \ldots, A_{\ell}'$ are arithmetic progressions with the same common difference. 
\end{theorem}

\section{Applications to subsequence sums}\label{applications-to-subseq-sum}
In this section, we assume that the group $G$ is written additively with additive identity $0$. Let $p(G)$ denote the order of the smallest nontrivial subgroup of a group $G$, or $\infty$ if no such subgroups exist. Let $\mathbf{a} = (a_1, \ldots, a_m)$ be a sequence of elements of $G$ (not necessarily abelian), where $m$ is a positive integer. Given a positive integer $\ell$ with $\ell \leq m$, recall that
\begin{equation*}\label{gen-subseq-sum-eq1}
  \Sigma^{\ell} (\mathbf{a}) = \{a_{i_1} + \cdots + a_{i_{\ell}}: i_j \in [1, m]~ \text{for}~ j = 1, \ldots, \ell\}.
\end{equation*}
Following the notation in \cite{devos2009}, we define
	\[\rho_a(\mathbf{a}) = |\{i \in [1, m] : a_i = a\}|.\]
We use the following notation also:
	\[\mu_{\mathbf{a}} (a) = \min \{\ell, \rho_a(\mathbf{a})\}.\]

It is easy to see that for each $a \in A$, we have 
\[\rho_a(\mathbf{a}) = \sum_{j = 1}^m \chi_{A_j}(a),\]
and
\[\mu_{\mathbf{a}} (a) = \mu_{\mathcal{A}}(a) = \min \bigg (\ell, \sum_{j = 1}^m \chi_{A_j}(a) \bigg).\]

The following result follows easily from Theorem \ref{productset-tf-dir-thm}.
\begin{theorem}\label{gen-subseq-sum-tf-dir-thm1}
Let $\ell$ and $m$ be positive integers with $\ell \leq m$. Let $\mathbf{a} = (a_1, \ldots, a_m)$ be a sequence of elements from a torsion-free group $G$, and let $A$ be the set of distinct terms of $\mathbf{a} = (a_1, \ldots, a_m)$. Then
\begin{equation}\label{gen-subseq-sum-tf-dir-thm-eq1}
 |\Sigma^{\ell} (\mathbf{a})| \geq \sum_{a \in A}\mu_{\mathbf{a}} (a) - \ell + 1
\end{equation}
\end{theorem}

Similary, in case of arbitrary groups, the following theorem follows easily from Theorem \ref{sigma-abl-dir-thm}.
\begin{theorem}\label{gen-subseq-sum-gp-dir-thm1}
Let $\ell$ and $m$ be positive integers with $\ell \leq m$. Let $\mathbf{a} = (a_1, \ldots, a_m)$ be a sequence of elements from a group $G$, and let $A$ be the set of distinct terms of $\mathbf{a} = (a_1, \ldots, a_m)$. Then
\begin{equation}\label{gen-subseq-sum-gp-dir-thm-eq1}
 |\Sigma^{\ell} (\mathbf{a})| \geq \min \biggl(p(G),~\sum_{a \in A}\mu_{\mathbf{a}} (a) - \ell + 1\biggr).
\end{equation}
\end{theorem}

Hamidoune \cite{hamidoune2003} proved the following result.
	
\begin{theorem}\label{hamidoune-susbseq-sum-thm}
	Let $\ell$ and $m$ be positive integers with $\ell \leq m \leq 2 \ell - 1$. Let $\mathbf{a} = (a_1, \ldots, a_m)$ be a sequence of elements from an abelian group $G$. Then one of the following holds:
	\begin{enumerate}
		\item $|\Sigma^{\ell} (\mathbf{a})| \geq m - \ell + 1$.
		\item There exists $i \in [1, m]$ such that $\ell a_i \in \Sigma^{\ell} (\mathbf{a})$.
	\end{enumerate}
\end{theorem}

Using Theorem \ref{gen-subseq-sum-tf-dir-thm1}, we prove the following anlogous theorem in torsion free groups.

\begin{theorem}\label{gen-subseq-sum-tf-dir-thm2}
Let $\ell$ and $m$ be positive integers with $\ell \leq m$. Let $\mathbf{a} = (a_1, \ldots, a_m)$ be a sequence of elements from a torsion-free group $G$. Then one of the following holds:
	\begin{enumerate}
		\item $|\Sigma^{\ell} (\mathbf{a})| \geq m - \ell + 1$.
		\item There exists $i \in [1, m]$ such that $\ell a_i \in \Sigma^{\ell} (\mathbf{a})$.
	\end{enumerate}
\end{theorem}

\begin{proof}
For $\ell = 1$, the result is trivial. Now assume that $\ell \geq 2$. If $\rho_{a_i}(\mathbf{a}) \geq \ell + 1$ for some $i \in [1, m]$, then $\ell a_i \in \Sigma^{\ell} (\mathbf{a})$. Now we assume that $\rho_{a_i}(\mathbf{a}) \leq \ell$ for each $i \in [1, m]$. In this case,
	\[\sum_{a \in A} \mu_{\mathbf{a}} (a) = |\{a_1\}| + \cdots + |\{a_m\}| = m.\]
Hence it follows from Theorem \ref{gen-subseq-sum-tf-dir-thm1} that
	\[|\Sigma^{\ell} (\mathbf{a})| \geq \sum_{a \in A} \mu_{\mathbf{a}} (a) - \ell + 1 = m - \ell + 1.\]
\end{proof}

Similarly, using Theorem \ref{gen-subseq-sum-gp-dir-thm1}, one can prove the following theorem in arbitrary groups.

\begin{theorem}\label{gen-subseq-sum-gp-dir-thm2}
Let $\ell$ and $m$ be positive integers with $\ell \leq m$. Let $\mathbf{a} = (a_1, \ldots, a_m)$ be a sequence of elements from a group $G$, and let $A$ be the set of distinct terms of $\mathbf{a} = (a_1, \ldots, a_m)$. Then one of the following holds:
	\begin{enumerate}
		\item $|\Sigma^{\ell} (\mathbf{a})| \geq \min \biggl(p(G),  m - \ell + 1\biggr)$.
		\item There exists $i \in [1, m]$ such that $\ell a_i \in \Sigma^{\ell} (\mathbf{a})$.
	\end{enumerate}
\end{theorem}

\begin{proof}
	The proof is similar to the proof of Theorem \ref{gen-subseq-sum-tf-dir-thm2}.
\end{proof}

The following theorem gives some structural information about the sequence. This theorem is analogous to Theorem \ref{productset-tf-inv-thm2}.
\begin{theorem}\label{gen-subseq-sum-tf-inv-thm1}
    Let $\ell$ and $m$ be integers such that $2 \leq \ell \leq m$. Let $\mathbf{a} = (a_1, \ldots, a_m)$ be a sequence of elements from a torsion-free group $G$, and let $A$ be the set of distinct terms of the given sequence. Let $\mu_{\mathbf{a}} (a) = \ell$ for some $a \in A$. Furthermore, assume that $|\{a \in A : \mu_{\mathbf{a}} (a) \geq 2\}| \geq 2$, and let
    \[|\Sigma^{\ell}(\mathbf{a})| =  \sum_{a \in A} \mu_{\mathbf{a}} (a) - \ell + 1.\]
    Then the set $A$ is an arithmetic progression.	
\end{theorem}

The following theorem is a further generalization of Theorem \ref{gen-subseq-sum-tf-inv-thm1}.	
	\begin{theorem}\label{gen-subseq-sum-tf-inv-thm2}
		Let $\ell$ and $m$ be positive integers such that $\ell \leq m - 2$. Let $\mathbf{a} = (a_1, \ldots, a_m)$ be a sequence of elements from a torsion-free group $G$, and let $X = \{a \in A : \mu_{\mathbf{a}}(a) \geq 2\}$. Let $|X| \geq 2$, and assume that there exists the smallest positive integer $t \geq 2$ such that
		\[\sum_{i = 1}^t \rho_{x_i}(\mathbf{a}) \geq \ell + t\]
		for some $x_1, \ldots, x_t \in X$. Let
		\[|\Sigma^{\ell}(\mathbf{a})| =  \sum_{a \in A} \mu_{\mathbf{a}} (a) - \ell + 1.\]
		Then the set $A$ of the distinct terms of the sequence $\mathbf{a} = (a_1, \ldots, a_m)$ is an arithmetic progression.
	\end{theorem}
	
	\begin{proof}
		Let $|X| = n$. Without loss of generality, by rearranging the terms of the sequence $\mathbf{a}$, we may assume that
		\[X = \{a_1, \ldots, a_n\},\]
		where
		\[\rho_{a_1}(\mathbf{a}) \geq \cdots \geq \rho_{a_n}(\mathbf{a}).\]
		Now we construct sets $X_1, \ldots, X_{\ell} \subseteq A$ satisfying the following conditions:
		\begin{equation}\label{gen-subseq-sum-tf-inv-thm2-eq1}
			X_1 = A,
		\end{equation}
\begin{equation}\label{gen-subseq-sum-tf-inv-thm2-eq2}
			\mu_{\bold{a}} (a) = \sum_{j = 1}^{\ell} \chi_{X_j}(a)~\text{for each}~ a \in A
		\end{equation}
		\begin{equation}\label{gen-subseq-sum-tf-inv-thm2-eq3}
			X_1 + \cdots + X_{\ell} \subseteq \Sigma^{\ell}(\mathbf{a}),
		\end{equation}
		\begin{equation}\label{gen-subseq-sum-tf-inv-thm2-eq4}
			|X_1| + \cdots + |X_{\ell}| = \sum_{a \in A} \mu_{\mathbf{a}} (a),
		\end{equation}
		and
		\begin{equation}\label{gen-subseq-sum-tf-inv-thm2-eq5}
			|X_1| \geq |X_2| \geq \cdots \geq |X_r| \geq 2,
		\end{equation}
		for some $r \in [2, \ell]$. Let $t \in [1, n]$ be the least integer such that 
		\[\sum_{i = 1}^t \rho_{a_i}(\mathbf{a}) \geq \ell + t\]
		and let
		\[k = \rho_{a_1}(\mathbf{a}) + \sum_{i = 2}^{t - 1}(\rho_{a_i}(\mathbf{a}) - 1) + (\rho_{a_t}(\mathbf{a}) - 2) \geq \ell.\]
Now consider the sequence $\mathbf{x} = (x_1, x_2, \ldots, x_k)$ defined as
\[\mathbf{x} = (\underbrace{a_1, \ldots, a_1}_{\rho_{a_1}(\mathbf{a})\ \text{times}}, \underbrace{a_2, \ldots, a_2}_{\rho_{a_2}(\mathbf{a}) - 1\ \text{times}}, \ldots, \underbrace{a_{t - 1}, \ldots, a_{t - 1}}_{\rho_{a_{t-1}}(\mathbf{a}) - 1\ \text{times}}, \underbrace{a_t, \ldots, a_t}_{\rho_{a_t}(\mathbf{a}) - 2 \ \text{times}}).\]  

Now, we define $A_i = \{x_i\}$ for each $i \in [1, \ell]$. Let $\mathbf{b}$ be the sequence obtained by deleting the terms $x_1, \ldots, x_{\ell}$ from $(a_1, \ldots, a_m)$, and let $B$ be set of all distinct elements of the sequence $\mathbf{b}$. Then it easy to see that 
		\[a_t \in B ~\text{and}~ \rho_{a_t} (\mathbf{b}) \geq 2.\]
		Now, we define the sets $A_1', \ldots, A_{\ell}'$ as follows
		\[A_j' = \{a \in B : \rho_{a} (\mathbf{b}) \geq j\}\]
		for each $j \in [1, \ell]$. Let 
		\[X_j = A_j \cup A_j'\]
		for each $j \in [1, \ell]$. Then it is easy to see that the sets $X_1, \ldots, X_{\ell}$ satisfies \eqref{gen-subseq-sum-tf-inv-thm2-eq1}, \eqref{gen-subseq-sum-tf-inv-thm2-eq2}, \eqref{gen-subseq-sum-tf-inv-thm2-eq3}, \eqref{gen-subseq-sum-tf-inv-thm2-eq4}, and \eqref{gen-subseq-sum-tf-inv-thm2-eq5}.
		Since $|\Sigma^{\ell}(\mathbf{a})| =  \sum_{a \in A} \mu_{\mathbf{a}}(a) - \ell + 1$, it follows from \eqref{gen-subseq-sum-tf-inv-thm2-eq3}, \eqref{gen-subseq-sum-tf-inv-thm2-eq4} and Lemma \ref{kemp-gen-thm-tf} that
		\begin{align*}
			\sum_{a \in A} \mu_{\mathbf{a}}(a) - \ell + 1 &= |\Sigma^{\ell}(\mathbf{a})|\\
			& \geq |X_1 + \cdots + X_{\ell}|\\
			&\geq |X_1| + \cdots + |X_{\ell}| - \ell + 1\\
			& = \sum_{a \in A} \mu_{\mathbf{a}}(a) - \ell + 1,
		\end{align*}
		and so
		\[|X_1 + \cdots + X_{\ell}| = |X_1| + \cdots + |X_{\ell}| - \ell + 1.\]
		Let $r \in [2, \ell]$ be the largest integer such that $|X_1| \geq |X_2| \geq \cdots \geq |X_r| \geq 2$. Then $|X_i| = 1$ for each $[1, \ell] \setminus [1, r]$. It is easy to see that
		\[|X_1 + \cdots + X_r| =|X_1 + \cdots + X_{\ell}|,\]
		and so
		\[|X_1 + \cdots + X_r| = |X_1| + \cdots + |X_{\ell}| - \ell + 1 = |X_1| + \cdots + |X_r| - r + 1.\]
		Since $|X_1 + \cdots + X_r| = |X_1| + \cdots + |X_r| - r + 1$, it follows from Lemma \ref{brail-gen-inv-thm-tf} that for each $i \in [1, r]$, the set $X_i$ is an arithmetic progression of type $(\alpha_i, g, \beta_i)$ for some $\alpha_i, \beta_i \in G$, where $\alpha_{i + 1} = - \beta_i$ and $g \neq 1$. Since $X_1 = A$, it follows that $A$ is an arithmetic progression. This completes the proof.
	\end{proof}

One can also prove the following theorem analogous to Theorem \ref{sigma-abl-inv-thm1}.
\begin{theorem}\label{gen-subseq-sum-ab-inv-thm1}
	Let $\ell$ and $m$ be integers such that $2 \leq \ell \leq m$. Let $\mathbf{a} = (a_1, \ldots, a_m)$ be a sequence of elements from an abelian group $G$, and let $A$ be the set of distinct terms of the given sequence. Let
	\begin{equation*}
		|\Sigma^{\ell}(\mathcal{A})| <
		\begin{cases}
			p(G) - 1, & \mbox{if } \ell = 2; \\
			p(G), & \mbox{if } \ell \geq 3.
		\end{cases}
	\end{equation*}
	Let $\mu (a) = \ell$ for some $a \in A$. Furthermore, $|\{a \in A : \mu_{\mathbf{a}}(a) \geq 2\}| \geq 2$, and let
	\[|\Sigma^{\ell}(\mathbf{a})| =  \sum_{a \in A} \mu_{\mathbf{a}} (a) - \ell + 1.\]
	Then the set $A$ is an arithmetic progression.	
\end{theorem}	

Using the similar argument as in the proof of Theorem \ref{gen-subseq-sum-tf-inv-thm2}, one can prove the following generalization of Theorem \ref{gen-subseq-sum-ab-inv-thm1}.	
	\begin{theorem}
		Let $\ell$ and $m$ be positive integers such that $\ell \leq m - 2$. Let $\mathbf{a} = (a_1, \ldots, a_m)$ be a sequence of elements from an abelian group $G$, and let $X = \{a \in A : \mu_{\mathbf{a}}(a) \geq 2\}$. Let $|X| \geq 2$, and assume that there exists the smallest positive integer $t \geq 2$ such that
		\[\sum_{i = 1}^t \rho_{x_i}(\mathbf{a}) \geq \ell + t\]
		for some $x_1, \ldots, x_t \in X$. Let
		\begin{equation*}
			|\Sigma^{\ell}(\mathcal{A})| =  \sum_{a \in A} \mu_{\mathbf{a}}(a) - \ell + 1 <
			\begin{cases}
				p(G) - 1, & \mbox{if } \ell = 2; \\
				p(G), & \mbox{if } \ell \geq 3.
			\end{cases}
		\end{equation*}
		Then $A$ is an arithmetic progression. 
	\end{theorem}
 	
% ------------------------------------------------------------------------

%\subsection*{Acknowledgment}
%Many thanks to our \TeX-pert for developing this class file.

% ------------------------------------------------------------------------

\appendix 

\section{Best possibility of lower bound in Theorem \ref{productset-tf-dir-thm}}\label{best-lower-bound}
Here we show that the lower bound in \eqref{productset-tf-dir-thm-eq1} is best possible. For this purpose, we construct the sequence $\mathcal{A} = (A_1, \ldots, A_m)$, where $A_1, \ldots, A_m$ are subsets of a cyclic subgroup $H$ of $G$ generated by a nonidentity element $g$. Since $G$ is a torsion-free group, it follows that $H$ is isomorphic to $\mathbb{Z}$. Therefore, it suffices to construct the sets $A_1, \ldots, A_m$ which are subsets of $\mathbb{Z}$, and which satisfy the following identity:
\begin{equation}\label{productset-tf-dir-thm-eq2}
	|\Sigma^{\ell}(\mathcal{A})| = \sum_{a \in A} \mu (a) - \ell + 1. 
\end{equation}
	
If $m = \ell$, then the identity \eqref{productset-tf-dir-thm-eq2} holds for the arithmetic progressions $A_1, \ldots, A_m$ which follows from the following special case of Lemma \ref{kemp-gen-thm-tf} and Lemma \ref{brail-gen-inv-thm-tf}. 
\begin{theorem}[\cite{nath, nath1}] \label{regular-hfold-direct-thm}
	Let $\ell \geq 2$. Let $A_1, \ldots, A_{\ell}$ be $\ell$ nonempty finite sets of integers. Then
\begin{equation}\label{regular-hfold-direct-thm-eq}
 |A_1 + \cdots + A_{\ell}| \geq |A_1| + \cdots + |A_{\ell}| - \ell + 1.
\end{equation}
The equality in \eqref{regular-hfold-direct-thm-eq} holds if and only if the sets $A_1, \ldots, A_{\ell}$ are arithmetic progressions with the same common difference.
\end{theorem}

Now assume that $2 \leq \ell < m$. We construct the sets $A_1, \ldots, A_m$ in each of the following cases:
\begin{enumerate}
  \item $\sum_{j = 1}^m \chi_{A_j}(a) \leq \ell~\text{for all}~ a \in A$.
  \item $\sum_{j = 1}^m \chi_{A_j}(a) \leq \ell~\text{for some}~ a \in A$, where $2 \leq \ell < m \leq 2 \ell$.
  \item $\sum_{j = 1}^m \chi_{A_j}(a) \leq \ell~\text{for some}~ a \in A$, where $m > 2 \ell$.
\end{enumerate}	
	
\noindent {\textbf{Construction 1.}}\label{case:1} Here we construct the sequence $\mathcal{A} = (A_1, \ldots, A_m)$ satisfying the lower bound in \eqref{productset-tf-dir-thm-eq1} and satisfying the following condition:
\[\sum_{j = 1}^m \chi_{A_j}(a) \leq \ell~\text{for all}~ a \in A.\]

Let $k_i = 0$ for $i \leq 0$, and let $k_1, \ldots, k_m$ be positive integers such that $k_1 \leq \cdots \leq k_m$. Let $n$ be a positive integer such that $n \ell \geq m$. Now for each $j \in [1, m]$, define
\[A_j = [k_{j - \ell} + \cdots + k_{j - n \ell} + 1, k_j + k_{j - \ell} + \cdots + k_{j - n \ell}].\]
Then $A_j = [1, k_j]$ for each $j \in [1, \ell]$, $|A_j| = k_j$ for each $j \in [1, m]$, and $\max(A_j) \leq \max(A_{j + 1})$ for each $j \in [1, m-1]$. Furthermore, it is easy to verify that
\[k_1 \leq \cdots \leq k_{\ell}  \leq \max(A_{\ell + 1}) \leq \cdots \leq \max(A_m),\]
and
\begin{equation}\label{productset-tf-dir-thm-eq3}
\max(A_{j - \ell}) = k_{j - \ell} + \cdots + k_{j - n \ell} < k_{j - \ell} + \cdots + k_{j - n \ell} + 1 = \min(A_j)
\end{equation}
for each $j \in [1, m]$.

Let $A_0$ be the empty set. Let $a \in A$ be an arbitrary element, and let $i_0 \in [1, m]$ be least integer such that $a \in A_{i_0}$. Then it follows from \eqref{productset-tf-dir-thm-eq3} that $a \in A_{i_0} \cup \cdots \cup A_{i_0 + \ell - 1}$ but $a \not\in A_{i_0 + 1} \cup \cdots \cup A_m$. Hence 
\[\sum_{j = 1}^m \chi_{A_j}(a) = \sum_{j = i_0}^{i_0 + \ell - 1} \chi_{A_j}(a)\leq \ell\]
which implies that $\mu(a) \leq \ell$.
 
It is easy to see that $\min (\Sigma^{\ell}(\mathcal{A})) = \ell$ and $\max (\Sigma^{\ell}(\mathcal{A})) = k_1 + \cdots + k_m$, and so
\[\Sigma^{\ell}(\mathcal{A}) \subseteq [\ell, k_1 + \cdots + k_m].\]
This implies that
\begin{equation}\label{productset-tf-dir-thm-eq4}
  |\Sigma^{\ell}(\mathcal{A})| \leq k_1 + \cdots + k_m - \ell + 1 =  \sum_{a \in A} \mu (a) - \ell + 1.
\end{equation}
It follows from \eqref{productset-tf-dir-thm-eq1} that
\begin{equation}\label{productset-tf-dir-thm-eq5}
	|\Sigma^{\ell}(\mathcal{A})| \geq  \sum_{a \in A} \mu (a) - \ell + 1. 
\end{equation}
Therefore, it follows from \eqref{productset-tf-dir-thm-eq4} and \eqref{productset-tf-dir-thm-eq5} that
\[|\Sigma^{\ell}(\mathcal{A})| =  \sum_{a \in A} \mu (a) - \ell + 1.\]	

\noindent {\textbf{Construction 2.}}\label{case:2} Here we construct the sequence $\mathcal{A} = (A_1, \ldots, A_m)$ satisfying the lower bound in \eqref{productset-tf-dir-thm-eq1} and satisfying the following condition:
\[\sum_{j = 1}^m \chi_{A_j}(a) > \ell~\text{for some}~ a \in A,\]
where $\ell < m \leq 2 \ell$.
 
Let $k_i = 0$ for $i \leq 0$, and let $k_1, \ldots, k_m$ be positive integers such that $2 \leq k_1 \leq \cdots \leq k_m$. Let $n_1, \ldots, n_{m - \ell}$ be positive integers such that $1 \leq n_j \leq k_j - 1$ for each $j \in [1, m - \ell]$ and $n_1 \geq n_2 \geq \cdots \geq n_{m - \ell}$ . Furthermore, for each $i \leq 0$, let $n_i = 0$. For each $j \in [1, m]$, define
\[A_j = [k_{j - \ell} - n_{j - \ell} + 1, k_j + k_{j - \ell} - n_{j - \ell}].\]
It is obvious that $A_i = [1, k_i]$ for each $i \in [1, \ell]$. Since $k_1 \leq \cdots \leq k_{\ell}$ and $k_1 \in A_{\ell + 1} = [k_1 - n_1 + 1, k_{\ell + 1} + k_1 - n_1]$, it follows that
\[\sum_{j = 1}^m \chi_{A_j}(k_1) \geq \sum_{j = 1}^{\ell + 1} \chi_{A_j}(k_1) = \ell + 1 > \ell.\]
It is easy to verify that
\[\min(\Sigma^{\ell}(\mathcal{A})) = \ell,\]
and
\[\max(\Sigma^{\ell}(\mathcal{A})) = k_1 + \cdots + k_m - n_1 - \cdots - n_{m - \ell}.\]
Hence
\[\Sigma^{\ell}(\mathcal{A}) \subseteq [\ell, k_1 + \cdots + k_m - n_1 - \cdots - n_{m - \ell}],\]
and so
\begin{equation}\label{productset-tf-dir-thm-eq6}
	|\Sigma^{\ell}(\mathcal{A})| \leq k_1 + \cdots + k_m - n_1 - \cdots - n_{m - \ell} - \ell + 1.
\end{equation}
Now for each $j \in [\ell + 1, m]$, define
\[B_j = [k_{j - \ell} + 1, k_j + k_{j - \ell} - n_{j - \ell}].\]
Let $\mathcal{B} =(A_1, \ldots, A_{\ell}, B_{\ell + 1}, \ldots, B_m)$ and $B = A_1 \cup \cdots \cup A_{\ell} \cup B_{\ell + 1} \cup \cdots \cup B_m$. Then it is easy to verify that $A = B$, $\Sigma^{\ell}(\mathcal{B}) \subseteq \Sigma^{\ell}(\mathcal{A})$, and
\[\sum_{j = 1}^{\ell} \chi_{A_j}(a) + \sum_{j = \ell + 1}^m \chi_{B_j}(a) \leq \ell\]
for all $ a \in B$. Therefore, with respect to the sequence $\mathcal{B}$, we have
\begin{align*}
	\mu_{\mathcal{B}} (a) = \min \bigg (\ell, \sum_{j = 1}^{\ell} \chi_{A_j}(a) + \sum_{j = \ell + 1}^m \chi_{B_j}(a) \bigg) = \sum_{j = 1}^{\ell} \chi_{A_j}(a) + \sum_{j = \ell + 1}^m \chi_{B_j}(a),
\end{align*}
and so
\begin{align*}
	\sum_{a \in B} \mu_{\mathcal{B}} (a) = \sum_{j = 1}^{\ell} |A_j| + \sum_{j = \ell + 1}^m |B_j| = k_1 + \cdots + k_m - n_1 - \cdots - n_{m - \ell}.
\end{align*}

Therefore, 
\begin{align}\label{productset-tf-dir-thm-eq7}
	|\Sigma^{\ell}(\mathcal{A})| \geq |\Sigma^{\ell}(\mathcal{B})| &\geq \sum_{a \in B}  \mu_{\mathcal{B}} (a) - \ell + 1 \notag \\
& = k_1 + \cdots + k_m - n_1 - \cdots - n_{m - \ell} - \ell + 1.
\end{align}
Therefore, it follows from \eqref{productset-tf-dir-thm-eq6} and \eqref{productset-tf-dir-thm-eq7} that
\[|\Sigma^{\ell}(\mathcal{A})| = k_1 + \cdots + k_m - n_1 - \cdots - n_{m - \ell} - \ell + 1.\]	

Also, it can be verified that with respect to the sequence $\mathcal{A}$, we have
\[\sum_{a \in A} \mu_{\mathcal{A}} (a) = k_1 + \cdots + k_m - n_1 - \cdots - n_{m - \ell}.\]
and so
\[|\Sigma^{\ell}(\mathcal{A})| = \sum_{a \in A} \mu_{\mathcal{A}} (a) - \ell + 1.\]

\noindent {\textbf{Construction 3.}}\label{case:3} Here we construct the sequence $\mathcal{A} = (A_1, \ldots, A_m)$ satisfying the lower bound in \eqref{productset-tf-dir-thm-eq1} and satisfying the following condition:
\[\sum_{j = 1}^m \chi_{A_j}(a) > \ell~\text{for some}~ a \in A,\]
where $m > 2 \ell$.
 
Let $k_i = 0$ for $i \leq 0$, and let $k_1, \ldots, k_m$ be integers such that $2 \leq k_1 \leq \cdots \leq k_m$. Let $n_1, \ldots, n_{\ell}$ be positive integers such that $1 \leq n_i \leq k_i - 1$ for each $i \in [1, \ell]$, and $n_1 \geq n_2 \geq \cdots \geq n_{\ell}$. For each $j \in [1, \ell]$, define
\[A_j = [1, k_j],\]
and for $j > \ell$, we define
\[A_j = [k_{j - \ell} + \cdots + k_{j - n \ell} - n_{j - n \ell} + 1, k_j + k_{j - \ell} + \cdots + k_{j - n \ell} - n_{j - n \ell}],\]
where $n = \lfloor j/\ell \rfloor$. It is easy to see that
\[\max(A_{i - 1}) \leq \max(A_{i})\]
for $i = 2, \ldots, m$, and
\[\max(A_{j - \ell}) < \min(A_j)\]
for $j >2\ell$. Also it can be verified that
\[\sum_{j = 1}^m \chi_{A_j}(k_1) \geq \ell + 1 > \ell,\]
and
\[\Sigma^{\ell}(\mathcal{A}) \subseteq [\ell , k_1 + \cdots + k_m - n_1 - \cdots - n_{\ell}],\]
and so
\begin{equation}\label{productset-tf-dir-thm-eq8}
	|\Sigma^{\ell}(\mathcal{A})| \leq k_1 + \cdots + k_m - n_1 - \cdots - n_{\ell} -\ell + 1.
\end{equation}
Now, for each $j \in [\ell + 1, 2 \ell]$, define
\[B_j = [k_{j - \ell} + 1, k_j + k_{j - \ell} - n_j]\]
Let 
\[\mathcal{B} = (A_1, \ldots, A_{\ell}, B_{\ell + 1}, \ldots, B_{2 \ell}, A_{2 \ell + 1}, \ldots, A_m)\]
and 
\[B = A_1 \cup \cdots \cup A_{\ell} \cup B_{\ell + 1} \cup \ldots \cup B_{2 \ell} \cup A_{2 \ell + 1} \cdots \cup A_m.\]
Then it is easy to verify that $A = B$, $\Sigma^{\ell}(\mathcal{B}) \subseteq \Sigma^{\ell}(\mathcal{A})$, and
\[\sum_{j = 1}^{\ell} \chi_{A_j}(a) + \sum_{j = \ell + 1}^{2 \ell} \chi_{B_j}(a) + \sum_{j = 2 \ell + 1}^m \chi_{A_j}(a)\leq \ell\]
for all $a \in B$. Therefore,
\[\sum_{a \in B} \mu_{\mathcal{B}} (a) = \sum_{j = 1}^{\ell} |A_i| + \sum_{j = \ell + 1}^{2 \ell} |B_i| + \sum_{j = 2 \ell + 1}^m |A_i|
= k_1 + \cdots + k_m - n_1 - \cdots - n_{\ell},\]
and so
\begin{equation}\label{productset-tf-dir-thm-eq9}
	|\Sigma^{\ell}(\mathcal{A})| \geq |\Sigma^{\ell}(\mathcal{B})| \geq  \sum_{a \in A} \mu_{\mathcal{B}} (a) - \ell + 1. k_1 + \cdots + k_m - n_1 - \cdots - n_{\ell} - \ell + 1.
\end{equation}
Therefore, it follows from \eqref{productset-tf-dir-thm-eq8} and \eqref{productset-tf-dir-thm-eq9} that
\[|\Sigma^{\ell}(\mathcal{A})| = k_1 + \cdots + k_m - n_1 - \cdots - n_{\ell} - \ell + 1.\]
Also, it can be verified that with respect to the sequence $\mathcal{A}$, we have
\[\sum_{a \in A} \mu_{\mathcal{A}} (a) = k_1 + \cdots + k_m - n_1 - \cdots - n_{m - \ell}.\]
and so
\[|\Sigma^{\ell}(\mathcal{A})| = \sum_{a \in A} \mu_{\mathcal{A}} (a) - \ell + 1.\]	

Thus we have constructed the sequence $\mathcal{A}$ in all possible cases which gives the lower bound in \eqref{productset-tf-dir-thm-eq1}.

%%%%%%%%%%%%%%%%%%%%%%%%%%%%%%%%%%

\end{document}